\documentclass[epsfig,11pt,reqno]{amsart}
\usepackage{amsthm,amssymb,amsmath,graphics,latexsym}
\usepackage{psfrag,graphicx,color,subfigure,comment,enumerate,mathtools}

\newtheorem{thm}{Theorem}[section]

\newtheorem{lem}[thm]{Lemma}
\newtheorem{prop}[thm]{Proposition}

\theoremstyle{definition}

\theoremstyle{remark}
\newtheorem{rem}[thm]{Remark}
\numberwithin{equation}{section}

\newtheorem{ass}{Assumption}

\begin{document}
\title[Linear convergence property of delayed gradient descent]{Non-ergodic linear convergence property of the delayed gradient descent under the strongly convexity and the Polyak-{\L}ojasiewicz condition} 
 
\author{Hyung Jun Choi}
\address[Hyung Jun Choi]{School of Liberal Arts, Korea University of Technology and Education, Cheonan 31253, Republic of Korea}
\email{hjchoi@koreatech.ac.kr; choihjs@gmail.com}

\author{Woocheol Choi}
\address[Woocheol Choi]{
Department of Mathematics, Sungkyunkwan University,  Suwon 16419,   Republic of Korea  }
\email{choiwc@skku.edu}

\author{Jinmyoung Seok}
\address[Jinmyoung Seok]{Department of Mathematics Education, Seoul National University,  Seoul 08826, Republic of Korea}
\email{jmseok@kgu.ac.kr}



\maketitle

\begin{abstract}
In this work, we establish the linear convergence estimate for the gradient descent involving the delay $\tau\in\mathbb{N}$ when the cost function is $\mu$-strongly convex and $L$-smooth. This result improves upon the well-known estimates in Arjevani et al. \cite{ASS} and Stich-Karmireddy \cite{SK} in the sense that it is non-ergodic and is still established in spite of weaker constraint of cost function. Also, the range of learning rate $\eta$ can be extended from $\eta\leq 1/(10L\tau)$ to $\eta\leq 1/(4L\tau)$ for $\tau =1$ and $\eta\leq 3/(10L\tau)$ for $\tau \geq 2$, where $L >0$ is the Lipschitz continuity constant of the gradient of cost function.
In a further research, we show the linear convergence of cost function under the Polyak-{\L}ojasiewicz\,(PL) condition, for which the available choice of learning rate is further improved as $\eta\leq 9/(10L\tau)$ for the large delay $\tau$. The framework of the proof for this result is also extended to the stochastic gradient descent with time-varying delay under the PL condition. Finally, some numerical experiments are provided in order to confirm the reliability of the analyzed results.
\end{abstract}

\section{Introduction}

In this paper, we consider the gradient descent with a fixed delay $\tau \in \mathbb{N}$ in the optimization problem:
\begin{equation}\label{eq-1-1}
x_{t+1}=\left\{\begin{aligned}
&x_t - \eta \nabla f(x_{t-\tau})&&\textrm{~for~}~t \geq \tau,\\
&x_0&&\textrm{~for~}~0\leq t <\tau,
\end{aligned}\right.
\end{equation}
where the given cost $f:\mathbb{R}^d \rightarrow \mathbb{R}$ is a differentiable convex function and the initial point $x_0 \in \mathbb{R}^d$ is given. Throughout this paper, denote by $x_*\in\mathbb{R}^d$ a unique minimizer of the cost function $f$.

\vspace{0.1cm}

The gradient descent scheme \eqref{eq-1-1} has gained a lot of interest from many researchers since it has a simple form among the gradient descent algorithms involving the delay which naturally appear in various learning problems such as the asynchronous gradient descent and the multi-agent optimization. So far, numerous related studies have been investigated. The delay effect in the stochastic gradient descent was studied in several literatures \cite{ASS, KSJ, LLXWZ, SK, SYL, XZYZC}, and moreover, some studies of the delay effect of communication in the decentralized optimization could be found in the references \cite{AD,SY, WX}.
Also, the online problem involving the feedback delay was studied in \cite{CB, JGS, LSZ, QK,  SS, WX, WTZ}. Understanding the convergence analysis of \eqref{eq-1-1} will be fundamentally important, as it serves as the foundation for numerous other related algorithms employed in the aforementioned applications. The delay also arises in other algorithms such as the coordinate descent method and the federated learning (refer to \cite{JWRBS, WYZF, YB, ZLLLH}).

\vspace{0.1cm}

This paper concerns the linear convergence property of \eqref{eq-1-1}.
We first recall two previous basic results of \cite{ASS} and \cite{SK} in this line of research. 
Arjevani et al. \cite{ASS} provided the following linear convergence result when the cost is given by a strongly convex quadratic function.

\begin{thm}[\cite{ASS}]Assume that $f:\mathbb{R}^d\rightarrow\mathbb{R}$ is $\mu$-strongly convex and $L$-smooth, and suppose that $f$ is given as $f(x) = 2^{-1}\,x^T Ax +b^T x +c$ for some $A\in\mathbb{R}^{d\times d}$, $b\in\mathbb{R}^d$ and $c\in\mathbb{R}$.
Then, for a positive stepsize $\eta\leq \frac{1}{20L(\tau+1)}$ it holds that 
 \begin{equation*}
 f(x_t) - f(x_*) \leq 5 L (1-\mu \eta)^{2t} \|x_0 -x_*\|^2\qquad \forall\, t \geq 0.
 \end{equation*}
\end{thm}

\vspace{0.1cm}

As for the general strongly convex and smooth cost functions,  
Stich and Karimireddy \cite{SK} showed the following ergodic type result:

\begin{thm}[\cite{SK}]Assume that $f:\mathbb{R}^d\rightarrow\mathbb{R}$ is $\mu$-strongly convex and $L$-smooth. 
There exists a positive stepsize $\eta\leq \frac{1}{10L\tau}$ such that
\begin{equation*}
{ \mathbb{E}  \Big(f(x^{out}) - f(x_*)\Big)} = O \Big( L \tau \|x_0 -x_*\|^2 \exp \Big[ - \frac{\mu T}{10L\tau}\Big]\Big)\qquad \forall\, T\geq 1,
\end{equation*}
where the expectation is given for the variable $x^{out} \in \{x_t\}_{t=0}^{T-1}$ which is chosen to be $x_t$ with probability proportional to $(1-\mu\eta /2)^{-t}$.
\end{thm} 

\vspace{0.1cm}

The aim of this paper is to improve the aforementioned results in following senses:
First, we obtain a non-ergodic linear convergence property of \eqref{eq-1-1} when $f$ is $\mu$-strongly convex and $L$-smooth.
Second, our linear convergence result permits a larger range of the stepsize $\eta >0$ than the results of \cite{ASS, SK}.

\vspace{0.1cm}

We first state the brief version of our main result with the linear convergence.

\begin{thm}\label{thm-1-1}
Assume that $f:\mathbb{R}^d\rightarrow\mathbb{R}$ is $\mu$-strongly convex and $L$-smooth. 
Then, for any choice of stepsize $\eta$ satisfying
\begin{equation}\label{eq-1-30}
0 < \eta \leq \frac{C_{\tau}}{L\tau},
\end{equation} 
the sequence $\{x_t\}_{t \geq 0}$ of \eqref{eq-1-1} satisfies the following estimate:
\begin{equation}\label{eq-thm-1-3}
\|x_t -x_*\| \leq \frac{1}{1-J_{\tau/2}L\tau \eta}\left(1-\frac{\eta\alpha}{ 2}\right)^{t/2} \|x_0 -x_*\|  \qquad \forall\,t\geq 0,
\end{equation}
where $C_{\tau} := \frac{\tau}{\sqrt{6J_{\tau}\tau^2 +1}+1}$ and $\alpha := \frac{2\mu L}{\mu+L}$ with a constant $J_n$ for $n \in \{k/2\mid k \in \mathbb{N}\}$ defined by
\begin{equation}\label{Jn-def}
J_n  \coloneqq \sup_{0 < x \leq \frac{1}{5}} \frac{1}{x}\left( \frac{1}{\left(1 - \frac{x}{n}\right)^n} -1\right).
\end{equation}
{ It is verified in Lemma \ref{lem-2-1} that the constant $J_n$ is finite and decreasing in $n$. 
Also, the supremum is obtained by $x = 1/5$ and satisfies the following bound:}
$J_n \leq 1.455$ for $n=1/2$, $J_n \leq 1.25$ for $n=1$, and $J_n \leq 1.2$ for $n\geq 3/2$. 
\end{thm}

\vspace{0.1cm}

\begin{rem}
We briefly describe their numeric values and boundness of some constants used in Theorem \ref{thm-1-1}.

\begin{enumerate}[$(i)$] 
\item The value of $C_\tau$ increases in $\tau$ and $\lim_{\tau \rightarrow \infty} C_{\tau} \leq 1/\sqrt{7.2} \simeq 0.3727$. 
In Table \ref{values}, we provide more numeric values of $C_\tau$ up to $\tau=8$ for the sake of convenience.  

\begin{table}[ht]\label{tab-5}
\centering
\begin{tabular}{|c|c|c|c|c|c|c|c|c| }\cline{1-9}
$\tau$& 1& 2 &3 &4&5&6&7&8
\\
\hline
&&&&&&&&\\[-1.2em]
$C_{\tau}$ &0.2553&0.3096   &0.3292 &0.3396 & 0.3459 & 0.3502 & 0.3534 & 0.3557
\\
[0.2em] 
\hline
\end{tabular}
\vspace{0.2cm}
\caption{Numeric values of $C_\tau$ up to $\tau = 8$}\label{values}
\end{table} 
\item We note that the coefficient $\frac{1}{1-J_{\tau/2}L\tau \eta}$ in \eqref{eq-thm-1-3} approaches to $1$ as $\eta \rightarrow 0$. Furthermore, since $J_{\tau/2} L \tau \eta \leq J_{\tau/2}C_{\tau} <1/2$ for all $\tau \geq 1$, one sees that
$$
0 < \frac{1}{1-J_{\tau/2}L\tau \eta} < 2\qquad\mbox{ for }~\eta \leq \frac{C_{\tau}}{L\tau}.
$$
\end{enumerate}
\end{rem}

\vspace{0.1cm}

Next, we give the generalized result of Theorem \ref{thm-1-1}.

\begin{thm}\label{thm-1-5}
Assume that $f:\mathbb{R}^d\rightarrow\mathbb{R}$ is $\mu$-strongly convex and $L$-smooth.
If we choose the stepsize $\eta>0$ satisfying
\begin{equation}\label{eq-3-20}
\eta \leq \min\Bigg\{\frac{L(1+q)}{5\alpha},~
 \frac{\tau}{\sqrt{2J_{\tau}\tau^2(2+1/q)+1}+1} \Bigg\} \frac{1}{\tau L},
\end{equation}
then the sequence $\{x_t\}_{t \geq 0}$ of \eqref{eq-1-1} satisfies
\begin{equation}\label{eq-2-1}
\|x_t -x_*\| \leq  \frac{1}{1-J_{\tau/2}L\tau \eta}   \left(1-\frac{\eta\alpha}{1+q}\right)^{t/2}  \|x_0 -x_*\|\qquad \forall\, t\geq 0,
\end{equation}
where $q>0$ is a chosen constant and $\alpha := \frac{2\mu L}{\mu+L}$.
In particular, since $\frac{L}{\alpha} \geq 1$ due to $L\geq \mu$, Theorem \ref{thm-1-1} is also satisfied with $q=1$.
\end{thm}  

\vspace{0.1cm}

In fact, the number $q>0$ in Theorem \ref{thm-1-5} appears in the progress of applying Young's inequality in the proof of Proposition \ref{prop-2-1}. Actually, we observe that if one choose a large $q>0$, then the range of $\eta$ in \eqref{eq-3-20} becomes larger with an upper bound while the diminishing factor in the decay estimate \eqref{eq-2-1} becomes close to one. Contrary, if we choose a small $q>0$, then the diminishing factor in \eqref{eq-2-1} becomes small but the range of $\eta$ in \eqref{eq-3-20} is more restricted.

\vspace{0.1cm}

For the proof of Theorem \ref{thm-1-5}, we will consider an auxiliary sequence $\{\tilde{x}_t\}$, which was also utilized in the previous work \cite{SK}. Using the smoothness and the strongly convexity of the function $f$, we will obtain a sequential inequality of the error $\|\tilde{x}_t -x_*\|$. We will then analyze the sequential inequality carefully to obtain a linear convergence estimate of the error $\|\tilde{x}_t -x_*\|$. So, it will be done by using Lemma \ref{lem-2-2} which is one of the main technical contributions of this paper. Afterwards, using the derived estimate of $\|\tilde{x}_t -x_*\|$ and measuring the difference between $x_t$ and $\tilde{x}_t$, we will finally show a linear convergence result of the error $\|x_t -x_*\|$.

\vspace{0.1cm}

On the other hand, we next derive a linear convergence result when the cost function satisfies the following Polyak-{\L}ojasiewicz\,(PL) condition: There exists a number $\zeta>0$ such that
\begin{equation}\label{eq-PL}
\frac{1}{2}\|\nabla f(x)\|^2 \geq \zeta \Big( f(x)-f_*\Big),
\end{equation}
where $f_*$ is the minimal value of $f$, i.e., $f_*=f(x_*)$.

\begin{thm}\label{thm-pl}
Suppose that $f:\mathbb{R}^d\rightarrow\mathbb{R}$ is $L$-smooth and it satisfies the PL condition \eqref{eq-PL}.
If we assume that the stepsize $\eta>0$ and the delay $\tau>0$ satisfy the following inequality:
\begin{equation}\label{eq-1-10}
\eta \leq \min\left\{ \frac{L}{5 \zeta},~ \frac{2\tau}{\sqrt{1+4 J_{\tau}  \tau^2}+1 }\right\}\frac{1}{L\tau},
\end{equation}
then for any $t \geq \tau$ we have
\begin{equation}\label{analyzed_PL}
 f(x_{t}) -f_* \leq \left(1-  {\eta \zeta} \right)^{t-\tau}\Big( f(x_\tau) -f_*\Big).
\end{equation}
\end{thm}

\vspace{0.1cm}

In order to prove Theorem \ref{thm-pl}, we will meet a sequence inequality in terms of the sequence $\{f(x_t)-f_*\}_{t \geq \tau}$ similar to the one appeared in the proof of Theorem \ref{thm-1-5}. We will then use the result of Lemma \ref{lem-2-1} to analyze the concerned sequence. Interestingly, its proof does not use the auxiliary sequence $\{\tilde{x}_t\}$ anymore, and so it is simpler than the proof of Theorem \ref{thm-1-5}. We also utilize the framework of the proof of Theorem \ref{thm-pl} to achieve a new convergence result for the stochastic gradient descent involving a time-varying delay in Section \ref{sgd}.

\vspace{0.1cm}

The rest of this paper is organized as follows. In Section \ref{sec-2}, we define an auxiliary sequence and show the linear convergence property of its error. With the help of this result, we derive the linearly convergent error estimate of the sequence $\{x_t\}_{t \geq 0}$ in Section \ref{sec-3}. Also, the linear convergence result under the PL condition is proved in Section \ref{sec-PL}. We establish a convergence result for the stochastic gradient descent involving a time-varying delay in Section \ref{sgd} under the PL condition. Finally, we provide some numerical examples in Section \ref{sec-4} for supporting our analyzed results.

\vspace{0.2cm}

\section{Preliminary convergence result }\label{sec-2}
 
In this section, we discuss a convergence result of an auxiliary sequence $\{\tilde{x}_t\}$ satisfying
\begin{equation}\label{eq-1-2}
\tilde{x}_0 = x_0,\qquad\tilde{x}_{t+1} = \tilde{x}_t - \eta \nabla f(x_t)\qquad \forall\, t\geq 0,
\end{equation}
where $\{x_t\}$ is defined by \eqref{eq-1-1}. It will play an essential role to obtain the linear convergence estimate \eqref{eq-2-1} (cf. \cite{SK}).
Throughout Section \ref{sec-2}-\ref{sec-3}, it is assumed that the cost function $f$ satisfies the following two assumptions:
\begin{ass}\label{LS}
The function $f$ is $L$-smooth, defined as follows: {\it There is a number $L >0$ such that}
\begin{equation}\label{L-smooth}
\| \nabla f(y) - \nabla f (x)\| \leq L\|y-x\|\qquad \forall\, x,\,y \in \mathbb{R}^d.
\end{equation}
\end{ass}
\begin{ass}\label{sc} 
The function $f$ is $\mu$-strongly convex, defined as follows: {\it There is a number $\mu >0$ such that}
\begin{equation}\label{strong}
f(y) \geq f(x) + (y-x)^T\cdot\nabla f(x) + \frac{\mu}{2}\|y-x\|^2\qquad\forall\,x,\,y \in \mathbb{R}^d.
\end{equation}
\end{ass}

We state an intermediate inequality regarding the quantity $\|\tilde{x}_t-x_*\|$.


\begin{prop}\label{prop-2-1}
Let $\tilde{x}_t$ be an auxiliary sequence defined by \eqref{eq-1-2}.
For the minimizer $x_*\in\mathbb{R}$ of the given cost function $f$, we obtain
\begin{equation}\label{eq-2-25}
\begin{split}
&\|\tilde{x}_{t+1} -x_*\|^2
\\ 
&\leq    \Big(1 - \frac{\alpha \eta}{ 1+q}\Big)   \|\tilde{x}_t -x_*\|^2 - \Big( \frac{\eta}{2L} -  \eta^2\Big) \|\nabla f(x_t)\|^2 +\eta^2 \tau\Big( \frac{\alpha \eta}{ q} + 2L \eta\Big)  \mathcal{R}(t), 
\end{split}
\end{equation}
where $\alpha := \frac{2\mu L}{\mu+L}$, $q>0$ is a constant and
\begin{equation}\label{Rt-def}
\mathcal{R}(t) := \left\{\begin{array}{ll}0& \textrm{ if }~t =0,
\\
 \sum_{j=0}^{t-1} \|\nabla f(x_j)\|^2 &\textrm{ if }~ 1 \leq t \leq \tau,
\\
\sum_{j=t-\tau}^{t-1} \|\nabla f(x_{j})\|^2 &\textrm{ if }~ t \geq \tau+1.
\end{array}\right.
\end{equation}

\end{prop}
\begin{proof}

The definition of $\tilde{x}_t$ in \eqref{eq-1-2} implies
\begin{equation}\label{xtilde-eq-2-3}
\begin{split}
&\|\tilde{x}_{t+1} -x_*\|^2
\\
 & = \|\tilde{x}_t -x_* - \eta \nabla f(x_t)\|^2
\\
& = \|\tilde{x}_t -x_*\|^2 - 2 \eta \langle \tilde{x}_t -x_*, \nabla f (x_t)\rangle + \eta^2 \|\nabla f(x_t)\|^2
\\
& =\|\tilde{x}_t -x_*\|^2 - 2 \eta \langle {x}_t -x_*, \nabla f (x_t)\rangle + \eta^2 \|\nabla f(x_t)\|^2 + 2 \eta \langle x_t - \tilde{x}_t, ~\nabla f(x_t)\rangle.
\end{split}
\end{equation}
By the strongly convexity and the $L$-smoothness property, we have
\begin{equation}\label{xt-eq-2-4}
\langle x_t -x_*, \nabla f(x_t)\rangle \geq \frac{\mu L}{\mu +L} \|x_t -x_*\|^2 + \frac{1}{\mu+L}\|\nabla f(x_t)\|^2.
\end{equation}
Applying \eqref{xt-eq-2-4} to \eqref{xtilde-eq-2-3}, one yields
\begin{equation}\label{xtilde-eq-2-5}
\begin{split}
&\|\tilde{x}_{t+1} -x_*\|^2 \\ 
& \leq \|\tilde{x}_t -x_*\|^2  -  {\alpha\eta}  \|x_t -x_*\|^2- \Big( \frac{2\eta}{\mu+L} -  \eta^2\Big) \|\nabla f(x_t)\|^2 + 2 \eta \langle x_t - \tilde{x}_t, ~\nabla f(x_t)\rangle.
\end{split}
\end{equation}
By Young's inequality, we easily note that
\begin{equation*}
2 \eta \langle x_t - \tilde{x}_t, ~\nabla f(x_t)\rangle \leq \frac{\eta}{2L} \|\nabla f(x_t)\|^2 +2L \eta \|x_t - \tilde{x}_t\|^2,
\end{equation*}
and again, applying this result to \eqref{xtilde-eq-2-5}, we have
\begin{equation}\label{eq-2-9}
\begin{split}
&\|\tilde{x}_{t+1} -x_*\|^2
\\ 
&\leq \|\tilde{x}_t -x_*\|^2  -  {\alpha \eta}  \|x_t -x_*\|^2- \Big( \frac{2\eta}{\mu+L} -  \eta^2 - \frac{\eta}{2L}\Big) \|\nabla f(x_t)\|^2+2L \eta \|x_t - \tilde{x}_t\|^2
\\
& \leq \|\tilde{x}_t -x_*\|^2  -  {\alpha \eta}  \|x_t -x_*\|^2- \Big( \frac{\eta}{2L} -  \eta^2\Big) \|\nabla f(x_t)\|^2+2L \eta \|x_t - \tilde{x}_t\|^2,
\end{split}
\end{equation}
where the last inequality is obtained by using $\frac{2 \eta}{\mu +L} \geq \frac{\eta}{L}$ due to $\mu \leq L$.
Since
\begin{equation*}
(1+q)\|x_t -x_*\|^2 + \Big(1 +\frac{1}{q}\Big) \|\tilde{x}_t -x_t\|^2 \geq \|\tilde{x}_t -x_*\|^2\qquad\mbox{ for }~q>0,
\end{equation*}
the inequality \eqref{eq-2-9} becomes
\begin{equation}\label{eq-2-10}
\begin{split}
&\|\tilde{x}_{t+1} -x_*\|^2 \\ 
& \leq \Big(1 - \frac{\alpha \eta}{ 1+q}\Big)   \|\tilde{x}_t -x_*\|^2 - \Big( \frac{\eta}{2L} -  \eta^2\Big) \|\nabla f(x_t)\|^2 + \Big( \frac{\alpha \eta}{ q} + 2L \eta\Big) \|x_t - \tilde{x}_t\|^2.
\end{split}
\end{equation}
On the other hand, subtracting \eqref{eq-1-2} from \eqref{eq-1-1}, one gets
\begin{equation*}
x_{t+1}-\tilde{x}_{t+1}=x_t-\tilde{x}_t-\eta\left(\nabla f(x_{t-\tau})-\nabla f(x_t)\right),
\end{equation*}
and then, this implies
\begin{equation*}
x_t-\tilde{x}_t= \left\{\begin{array}{ll} 0 &\textrm{if }~ t= 0,
\\
\eta \sum_{j=0}^{t-1} \nabla f(x_j) &\textrm{if }~ 0 \leq t \leq \tau,
\\
 \eta \sum_{j=t-\tau}^{t-1} \nabla f(x_{j})&\textrm{if }~t \geq \tau +1.
 \end{array}
 \right.
\end{equation*}
By Cauchy-Schwartz inequality, we have
\begin{equation}\label{xt-xtilde-eq-2-9}
\|x_t - \tilde{x}_t\|^2   \leq \eta^2 \tau \mathcal{R}(t),
\end{equation}
where $\mathcal{R}(t)$ is given by \eqref{Rt-def}.
Applying \eqref{xt-xtilde-eq-2-9} to \eqref{eq-2-10}, the desired inequality \eqref{eq-2-25} is shown.
\end{proof}

\vspace{0.2cm}

For simplicity, we let
\begin{equation*}
a_t =  \|\tilde{x}_t -x_*\|^2 , \qquad b_t = \|\nabla f(x_t)\|^2.
\end{equation*}
Then the result \eqref{eq-2-25} in Proposition \ref{prop-2-1} is rewritten as
\begin{equation}\label{at-eq-2-9}
a_{t+1} \leq  \Big(1 - \frac{\alpha \eta}{ 1+q}\Big)a_t  - \Big( \frac{\eta}{2L} -  \eta^2\Big) b_t  + \eta^2 \tau\Big( \frac{\alpha \eta}{ q} +2L \eta\Big) \mathcal{R}(t).
\end{equation}
Since $\alpha \leq L$, the inequality \eqref{at-eq-2-9} becomes
\begin{equation}\label{eq-2-26}
\begin{split}
a_{t+1}  \leq & ~\Big(1 - \frac{\alpha \eta}{ 1+q}\Big)a_t  - \Big( \frac{\eta}{2L} -\eta^2\Big) b_t  
\\
&\qquad ~+ \left\{ \begin{array}{ll} 0 & \textrm{for }~t =0,
\\
\delta (b_{t-1} + \cdots + b_0)&\textrm{for }~ 1 \leq t \leq \tau,
\\
\delta (b_{t-1} + \cdots + b_{t-\tau})&\textrm{for }~t \geq \tau +1,
\end{array}
\right.
\end{split}
\end{equation}
where $\delta := \left( 2+ \frac{1}{ q}\right) \eta^3 \tau L$.

\vspace{0.2cm}

From \eqref{eq-2-26}, we next establish the estimate of $a_t$ in the following lemma.

\begin{lem}\label{lem-2-2}
Suppose that two positive sequences $\{a_t\}_{t \geq 0}$ and $\{b_t\}_{t \geq 0}$ satisfy
\begin{equation}\label{eq-2-20}
a_{t+1}  \leq \left\{\begin{array}{ll}c a_0 - Qb_0 &\quad \textrm{for } ~t=0,
\\
ca_t +\delta ( b_{t-1} + \cdots + b_0) - Qb_t&\quad \textrm{for }~ 1 \leq t \leq \tau,
\\
c a_t + \delta (b_{t-1} + \cdots + b_{t-\tau}) - Q b_t, &\quad \textrm{for }~t \geq \tau+1,
\end{array}\right.
\end{equation}
where $c$, $\delta$ and $Q$ are some positive constants.
If we assume that for any $1\leq j \leq \tau$,
\begin{equation}\label{eq-2-21}
\sum_{k=0}^{j-1}   c^{k} \delta \leq  c^{j} Q,
\end{equation}
then we have
\begin{equation}\label{at-eq-2-14}
a_{t+1} \leq c^{t+1} a_0\qquad\forall\, t \geq 0.
\end{equation}
\end{lem}

\begin{proof}

In this proof, we will show the desired result \eqref{at-eq-2-14} separately considering three cases: $t=0$, $1\leq t\leq \tau$ and $t\geq \tau+1$.

\vspace{0.1cm}

\noindent{\underline{\it Case 1.}~}
When $t=0$, the inequality \eqref{eq-2-20} implies
\begin{equation*}
a_1 \leq c a_0 - Qb_0\leq c a_0,
\end{equation*}
which directly gives the desired one \eqref{at-eq-2-14} for $t=0$.

\vspace{0.1cm}

\noindent{\underline{\it Case 2.}~}
We next show \eqref{at-eq-2-14} for $1\leq t\leq\tau$.
By \eqref{eq-2-20} for $t=0$ and $1$, one has
\begin{equation*}
\begin{split}
a_2 & \leq ca_1 + \delta b_0 -Qb_1 
\\
& \leq c (ca_0 - Qb_0) + \delta b_0 - Qb_1
\\
& = c^2 a_0 + (\delta -cQ)b_0 -Qb_1.
\end{split}
\end{equation*}
Generally, using \eqref{eq-2-20} for $0\leq t\leq\tau$ sequentially, we can obtain that for $1 \leq t \leq \tau$,
\begin{equation}\label{eq-2-18}
\begin{aligned}
a_{t+1}  &\leq \, c^{t+1} a_0 + \sum_{j=0}^{t-1} \Big[ \delta  + c \delta + \cdots + c^{t-1-j} \delta - c^{t-j} Q\Big] b_j - Qb_t\\
&= \, c^{t+1} a_0 + \sum_{j=0}^{t-1}\left( \sum_{k=0}^{t-1-j}c^k\delta - c^{t-j} Q\right) b_j - Qb_t.
\end{aligned}
\end{equation}
This inequality with the condition \eqref{eq-2-21} implies that $a_{t+1} \leq c^{t+1} a_0$ for $1 \leq t \leq \tau$.

\vspace{0.1cm}

\noindent{\underline{\it Case 3.}~}
For $t \geq \tau+1$, we first claim that
\begin{equation}\label{eq-2-19}
\begin{split}
a_{t+1}& \leq \, c^{t+1} a_0 + \sum_{j=t-\tau}^{t-1} \Big[\delta + c \delta + \cdots + c^{t-1-j}\delta - c^{t-j}Q\Big] b_j
\\
&\qquad + \sum_{j=0}^{t-1-\tau} c^{t-\tau-j}\Big[ \delta + c\delta + \cdots + c^{\tau-1} \delta -c^{\tau}Q\Big] b_j - Qb_t.
\end{split}
\end{equation}
To show \eqref{eq-2-19}, we try to use the mathematical induction.\\
{\underline{\it Base case.}~}
By \eqref{eq-2-20} with $t=\tau+1$ and \eqref{eq-2-18} with $t=\tau$, one has
\begin{equation}\label{atautwo-eq-2-17}
\begin{split}
a_{\tau +2}& \leq \, c a_{\tau+1}  +  \delta (b_{\tau}+b_{\tau-1} + \cdots + b_{1}) - Q b_{\tau+1}
\\
& \leq \, c\Big(c^{\tau+1} a_0 + \sum_{j=0}^{\tau-1} \Big[ \delta  + c \delta + \cdots + c^{\tau-1-j} \delta - c^{\tau-j} Q\Big] b_j - Qb_{\tau}\Big)
\\
&\qquad  +  \delta (b_{\tau} + b_{\tau-1} + \cdots + b_{1}) - Q b_{\tau+1}.
\end{split}
\end{equation}
Rearranging the right hand side in \eqref{atautwo-eq-2-17}, we get
\begin{equation*}
\begin{split}
a_{\tau+2}& \leq \, c^{\tau +2}a_0 + \sum_{j=1}^{\tau-1} \Big[ \delta + c \delta + \cdots + c^{\tau -1 -j} \delta - c^{\tau-j} Q\Big] b_j +(\delta -  cQ)b_{\tau}  - Qb_{\tau+1}
\\
&\qquad + c\Big[ \delta + c \delta + \cdots + c^{\tau -1} \delta - c^{\tau} Q\Big] b_0
\\
& = \, c^{\tau +2}a_0 + \sum_{j=1}^{\tau} \Big[ \delta + c \delta + \cdots + c^{\tau -1 -j} \delta - c^{\tau-j} Q\Big] b_j   
\\
&\qquad + c\Big[ \delta + c \delta + \cdots + c^{\tau -1} \delta - c^{\tau} Q\Big] b_0  - Qb_{\tau+1}
\end{split}
\end{equation*}
which corresponds to the inequality \eqref{eq-2-19} for $t = \tau+1$.

\noindent{\underline{\it Induction step.}~}
Assume that \eqref{eq-2-19} holds for $t=k$ with $k \geq \tau +1$.
By applying \eqref{eq-2-20} with $t=k+1$ and \eqref{eq-2-19} with $t=k$ in order, we have
 \begin{equation}\label{attwo-eq-2-18}
 \begin{split}
 a_{k+2}& \leq \, ca_{k+1} + \delta (b_k + \cdots + b_{k+1-\tau}) - Qb_{k+1}
 \\
 & \leq \, c^{k+2} a_0 + \sum_{j=k-\tau}^{k-1} \Big[ c \delta + c^2 \delta + \cdots + c^{k-j} \delta - c^{k+1-j} Q\Big] b_j 
 \\
 &\qquad + \sum_{j=0}^{k-1-\tau} c^{k+1 -\tau -j} \Big[ \delta + c \delta + \cdots + c^{\tau-1} \delta - c^{\tau} Q\Big] b_j
 \\
 &\qquad + \delta (b_k + \cdots + b_{k+1-\tau}) - Q b_{k+1}.
 \end{split}
 \end{equation}
Rearraning the right hand side in \eqref{attwo-eq-2-18}, we obtain
 \begin{equation*}
 \begin{split}
 a_{k+2} & \leq \, c^{k+2} a_0 + \sum_{j=k+1-\tau}^{k-1} \Big[ \delta + c \delta + \cdots + c^{k-j} \delta -c^{k+1-j} Q\Big] b_j 
 \\
 &\qquad + (\delta b_k - cQb_k)  - Qb_{k+1}
 \\
 &\qquad + \sum_{j=0}^{k-\tau} c^{k+1-\tau -j}\Big[ \delta + c \delta + \cdots + c^{\tau-1}\delta -c^{\tau} Q\Big] b_j
 \\
 & = \, c^{k+2} a_0 + \sum_{j=k+1-\tau}^{k} \Big[ \delta + c \delta + \cdots + c^{k-j} \delta -c^{k+1-j} Q\Big] b_j 
 \\
 &\qquad + \sum_{j=0}^{k-\tau} c^{k+1-\tau -j}\Big[ \delta + c \delta + \cdots + c^{\tau-1}\delta -c^{\tau} Q\Big] b_j
 -   Qb_{k+1},
 \end{split}
 \end{equation*}
 which is the desired inequality \eqref{eq-2-19} for $t=k+1$.
 
By the mathematical induction, we therefore obtain that \eqref{eq-2-19} holds for any $t \geq \tau+1$.
Given the estimate \eqref{eq-2-19}, it follows directly from \eqref{eq-2-21} that $a_{t+1} \leq c^{t+1} a_0$ for any $t \geq \tau+1$.
\end{proof}

\vspace{0.2cm}

Hopefully, we will use Lemma \ref{lem-2-2} for the sequences $\{a_t\}$ and $\{b_t\}$ satisfying \eqref{eq-2-26}. So, the following lemma will be required in order to check \eqref{eq-2-21} with
\begin{equation}\label{cQdelta-eq-2-19}
c=1-\frac{\alpha\eta}{1+q}>0,\qquad Q=\frac{\eta}{2L}-\eta^2>0,\qquad \delta=\left(2+\frac{1}{q}\right)\eta^3\tau L>0.
\end{equation}

{\begin{lem}\label{lem-2-1} 
Define
\[
f_k(x)\coloneqq \frac{1}{x} \left(\left(1 - \frac{x}{k/2}\right)^{-k/2} -1\right)\qquad\mbox{ for }\,k = 1,\,2,\,3,\,\dots.
\]
Then with $n = k/2$, we have
\[ 
J_n \coloneqq \sup_{x\in (0,1/5]}f_k(x)  = f_k(1/5).
\]
In addition, the value of $J_n$ decreases in $n\geq 1/2$ and satisfies the following bound:
\begin{equation}\label{eq-2-31}
J_n \leq\left\{\begin{array}{ll} 1.455& \quad \textrm{for }~n=1/2,
\\
1.25 &\quad \textrm{for }~ n=1,
\\ 1.2& \quad \textrm{for } ~n\geq 3/2.
\end{array}
\right.
\end{equation}

\end{lem}
\begin{proof}
We first show that $f_k$ is increasing in $x \in (0,\,1/5]$. Observe that
\[
\begin{aligned}
f_k(x) &= \frac{1-\left(1 - \frac{2x}{k}\right)^{k/2}}{x\left(1 - \frac{2x}{k}\right)^{k/2}} = \frac{\left(1-\left(1 - \frac{2x}{k}\right)^{1/2}\right)\sum_{j=0}^{k-1}\left(1 - \frac{2x}{k}\right)^{j/2}}{x\left(1 - \frac{2x}{k}\right)^{k/2}} \\
& = \frac{\left(1-\left(1 - \frac{2x}{k}\right)\right)\sum_{j=0}^{k-1}\left(1 - \frac{2x}{k}\right)^{(j-k)/2}}{x\left(1+\left(1 - \frac{2x}{k}\right)^{1/2}\right)} \\
& = \frac{2}{k}\left(1+\left(1 - \frac{2x}{k}\right)^{1/2}\right)^{-1}\sum_{j=0}^{k-1}\left(1 - \frac{2x}{k}\right)^{(j-k)/2}.
\end{aligned}
\]
Then differentiating it, we easily see that $f'_k(x) > 0$ for $x\in (0,\,1/5]$ so that $J_n = f_k(1/5)$.

In order to show that $J_n$  decreases in $n$, it suffices to show that the value of $\left(1- \frac{1}{5n}\right)^n$ increases in $n$. 
To show this, we introduce a function $H:[1/2,\,\infty) \rightarrow \mathbb{R}$ defined by 
\begin{equation*}
H(y) =\ln\left( 1- \frac{1}{5y}\right)^y = y \ln \left(1-\frac{1}{5y}\right).
\end{equation*}
We note that  $H' (y) = \ln \left(1- \frac{1}{5y}\right) + \frac{1}{5y-1}$ and 
\begin{equation*}
H'' (y) = \frac{1}{y(5y-1)} - \frac{1}{(y-1/5)(5y-1)} < 0\qquad \forall\, y \geq 1/2.
\end{equation*}
Combining this fact with that $\lim_{y \rightarrow \infty} H'(y) = 0$, we find that $H' (y) >0$ for $y \in [1/2, \infty)$. Therefore, $H(y)$ increases in $y \geq 1/2$, and so $J_n$ decreases in $n \geq 1/2$.

Finally, we have the bound \eqref{eq-2-31} by numerical computation for $f_k(1/5)$ with $k = 1,\,2,\,3$.
\end{proof}
\begin{rem}By Lemma \ref{lem-2-1}, we have
\begin{equation}\label{eq-2-30}
\frac{1}{\left(1 - \frac{x}{n}\right)^n} \leq 1 + J_n x \qquad \forall\, n\in \{k/2 \mid k \in \mathbb{N}\}, ~x \in (0,\,1/5].
\end{equation}
\end{rem}
}

\vspace{0.2cm}

Now, we show that the error sequence $\|\tilde{x}_t -x_*\|$ decays exponentially fast as $t$ goes to infinity in the following theorem.

\begin{thm}\label{thm-2-1}
Let $\tilde{x}_t$ be an auxiliary sequence defined by \eqref{eq-1-2} and $x_*\in\mathbb{R}$ be the minimizer of the given cost function $f$.
For a given $q>0$, if we choose $\eta>0$ satisfying
\begin{equation}\label{eq-2-28}
\eta \leq \min\Bigg\{ { \frac{L (1+q)}{5\alpha}},~{  \frac{\tau}{\sqrt{2J_{\tau}\tau^2(2+1/q)+1}+1}} 
 \Bigg\} \frac{1}{\tau L},
\end{equation}
then we have
\begin{equation*}
\|\tilde{x}_t -x_*\|^2 \leq \left(1 - \frac{\alpha \eta}{ 1+q}\right)^{t} \|x_0 -x_*\|^2.
\end{equation*}
\end{thm}
\begin{proof}
To prove the exponential decaying property of $a_t=\|\tilde{x}_t -x_*\|^2$ satisfying the sequential inequality \eqref{eq-2-26}, we will use Lemma \ref{lem-2-2}.
So, it is suffices to show that \eqref{eq-2-21} in Lemma \ref{lem-2-2} holds for $c$, $Q$ and $\delta$ given in \eqref{cQdelta-eq-2-19}.
The condition \eqref{eq-2-21} of Lemma \ref{lem-2-2} is equivalent to
\begin{equation*}
\delta \left(\frac{1-c^{j}}{1-c}\right) \leq Q c^j \qquad \forall \, j=1,2,\cdots,\tau,
\end{equation*}
which is rearranged to
\begin{equation}\label{cq-eq-2-25}
\frac{c_q \eta^2 \tau L}{\alpha} \leq \left( Q + \frac{c_q \eta^2 \tau L}{\alpha}\right) c^j \qquad\forall \, j=1,2,\cdots,\tau,
\end{equation}
where $c_q :=  (1+q) \left( 2 + \frac{1}{ q}\right)$.
Since the right hand side of \eqref{cq-eq-2-25} decreases in $j$ due to $c<1$, it is sufficient to show \eqref{cq-eq-2-25} for only $j=\tau$, which is the same with
\begin{equation}\label{eq-2-6}
  c^{-\tau}\leq  1 + \frac{Q\alpha}{c_q \eta^2 \tau L}.
\end{equation}
Since $\alpha \leq L$ and by \eqref{eq-2-28}, we note that
\begin{equation*}
\frac{\alpha \eta \tau}{ 1+q}  \leq \frac{1}{5},
\end{equation*}
and then Lemma \ref{lem-2-1} for $x=\frac{\alpha\eta\tau}{1+q}$ and $n=\tau$ gives
\begin{equation*}
c^{-\tau} =\frac{1}{\left(1 - \frac{\alpha \eta}{ 1+q}\right)^{\tau}} \leq 1 +J_{\tau} \left(\frac{\alpha \eta \tau}{ 1+q}\right).
\end{equation*}
Thus, to show \eqref{eq-2-6}, it remains to check
\begin{equation*}
1+J_{\tau} \left(\frac{\alpha \eta \tau}{ 1+q}\right) \leq 1 + \frac{Q\alpha}{c_q \eta^2 \tau L} = 1+ \frac{\alpha}{c_q { \eta} \tau L} \left( \frac{1}{2L} -\eta\right),
\end{equation*}
which is reduced to
\begin{equation}\label{tau-1q-eq-2-27}
\frac{\tau}{ 1+q}  \leq \frac{1}{ J_{\tau} c_q   \eta^2 \tau L} \left( \frac{1}{2L} -\eta\right).
\end{equation}
Actually, the inequality \eqref{tau-1q-eq-2-27} is equivalent to
\begin{equation}
\eta^2 + \eta\left(\frac{ 1+q}{ J_{\tau} c_q \tau^2 L}\right)  \leq \frac{1+q}{2 J_{\tau} L^2 \tau^2 c_q},
\end{equation}
which holds true for $\eta >0$ satisfying
\begin{equation*}
\begin{split}
\eta\, &\leq \, \sqrt{ \frac{1+q}{2 J_{\tau} L^2 \tau^2 c_q} + \frac{(1+q)^2}{ 4J_{\tau}^2 c_q^2 \tau^4 L^2}} - \frac{1+q}{ 2J_{\tau} c_q \tau^2 L}
\\
&=\,\frac{1}{\left(\sqrt{2J_{\tau}\tau^2(2+1/q)+1}+1\right)L}.
\end{split}
\end{equation*}
Under the assumption \eqref{eq-2-28}, the desired inequality \eqref{eq-2-6} is shown, and then the proof is concluded.
\end{proof}

\vspace{0.2cm}

\section{Proof of Theorem \ref{thm-1-5}}\label{sec-3}

In this section, we derive a decay estimate of the gradient descent with a fixed delay $\tau\in\mathbb{N}$.
This main result of this paper will be essentially obtained by the use of decaying property of the auxiliary sequence $\{\tilde{x}_k\}$ proved in Theorem \ref{thm-2-1}.

\vspace{0.1cm}

\begin{proof}[Proof of Theorem \ref{thm-1-5}]
~This proof will be shown by using the inductive argument.\\
{\underline{\it Base case.}~}
By Lemma \ref{lem-2-1} with $x=\frac{\eta\alpha(\tau/2)}{1+q}\leq \frac{\tau \eta L}{2(1+q)} \leq \frac{1}{10}$ and $n=\tau/2$, and since $\alpha\leq L$, one yields
\begin{equation}\label{R-tau-eq-3-1}
\begin{split}
\frac{1}{\left( 1- \frac{\eta \alpha}{1+q}\right)^{\tau /2}}& \leq 1 + J_{\tau/2} \left(\frac{\eta \alpha (\tau/2)}{1+q}\right)
\\
&< 1+ J_{\tau/2} { L \tau\eta}
\\
&\leq \frac{1}{1-J_{\tau/2} L\tau \eta},
\end{split}
\end{equation}
which is equivalent to
\begin{equation}\label{B-eq-3-2}
B:=\frac{1}{1 -J_{\tau/2}L\tau \eta} \left( 1- \frac{\eta \alpha}{1+q}\right)^{\tau /2} \geq 1.
\end{equation}
Since $x_{t} = x_0$ for $ 0 \leq t \leq \tau$ and by \eqref{B-eq-3-2}, one sees that \eqref{eq-2-1} obviously holds for $ 0\leq t \leq \tau$.

\noindent{\underline{\it Induction step.}~}
Let $k \geq \tau+1$. Suppose that \eqref{eq-2-1} holds for $0 \leq t \leq k-1$, i.e.,
\begin{equation}\label{eq-2-3}
\|x_t -x_*\| \leq B\left(1-\frac{\eta \alpha}{1+q}\right)^{(t-\tau)/2} \qquad\mbox{ for }\, 0 \leq t \leq k-1.
\end{equation}
Based on \eqref{eq-2-3}, we will show that \eqref{eq-2-1} holds true for $t = k$.
Since
$$
x_k=\tilde{x}_k - \eta \sum_{j=1}^{\tau}\left( \nabla f(x_{k-j})-\nabla f(x_*)\right)
$$
and by the triangle inequality and the $L$-smooth property, we have
\begin{equation}\label{eq-3-10}
\begin{split}
\|x_k -x_* \| &  = \Big\| \tilde{x}_k -x_* - \eta \sum_{j=1}^{\tau} \left(\nabla f(x_{k-j})-\nabla f(x_*)\right)\Big\|
\\
&\leq \| \tilde{x}_k -x_*\|+ \eta \sum_{j=1}^{\tau}\| \nabla f(x_{k-j})-\nabla f(x_*)\|\\
& \leq \|\tilde{x}_k -x_*\| + \eta L \sum_{j=1}^{\tau} \|x_{k-j} -x_*\|.
\end{split}
\end{equation}
Applying Theorem \ref{thm-2-1} and \eqref{eq-2-3} to \eqref{eq-3-10}, one gets
\begin{equation}\label{eq-2-5}
 \|x_k -x_* \| 
 \leq  R^k \|x_0 -x_*\| + \eta L B \sum_{j=1}^{\tau} R^{(k-\tau-j)},
\end{equation}
where $R := \Big(1-\frac{\eta \alpha}{ 1+q}\Big)^{1/2}$.
Here, we note that
\begin{equation}\label{R-eq-3-6}
\begin{split}
 \sum_{j=1}^{\tau} R^{(k-\tau-j)}& =R^{k-\tau-1} \sum_{j=1}^{\tau} R^{-j+1}
 \\
 & = R^{k-\tau-1} \left(\frac{R^{-\tau}-1}{R^{-1}-1}\right)
 \\
 & = R^{k-\tau} \left(\frac{R^{-\tau}-1}{1-R}\right).
 \end{split}
\end{equation} 
By \eqref{R-eq-3-6}, the inequality \eqref{eq-2-5} becomes
\begin{equation}\label{eq-2-7}
\|x_k -x_* \|\leq  R^{k-\tau} \left[R^{\tau}\|x_0 -x_*\|   + \eta LB\left(\frac{R^{-\tau}-1}{1 -R}\right)\right].
\end{equation}
Actually, the previous result \eqref{R-tau-eq-3-1} is rewritten as
\begin{equation*}
R^{-\tau} -1 \leq  J_{\tau/2}\left( \frac{\tau \eta \alpha}{2(1+q)}\right),
\end{equation*}
which gives
\begin{equation}\label{eta-LB-eq-3-8}
\begin{split}
\eta LB\left(\frac{R^{-\tau}-1}{1 -R}\right) & = \eta LB\left(\frac{(R^{-\tau}-1)(1+R)}{1 -R^2}\right) 
\\
& \leq \eta LB \left(\frac{ 1+q}{\eta\alpha}\right) \cdot J_{\tau/2} \left( \frac{ \tau \eta \alpha}{2(1+q)} \right)\cdot (1+R)
\\
& \leq \tau\eta LB  {J_{\tau/2}}.
\end{split}
\end{equation}
Applying \eqref{eta-LB-eq-3-8} to \eqref{eq-2-7}, we have
\begin{equation}\label{eq-2-8}
\|x_k -x_* \|\leq  R^{k-\tau} \left( R^{\tau} \|x_0 -x_*\|   +  B  J_{\tau/2}(L\tau \eta) \right).
\end{equation}
From the definition \eqref{B-eq-3-2} of $B$, one sees the following identity:
\begin{equation}\label{B-identity-eq-3-10}
R^{\tau}\|x_0 -x_*\|  +  B  {J_{\tau/2}(L\tau \eta)}  = B.
 \end{equation} 
Combining \eqref{eq-2-8} with \eqref{B-identity-eq-3-10}, we obtain
\begin{equation*}
\|x_k -x_* \| \leq   B \left(1-\frac{\eta \alpha}{ 1+q}\right)^{(k-\tau)/2},
\end{equation*}
which is the same with \eqref{eq-2-1} for $t=k$, and so the inductive arugment is complete.
\end{proof}

\vspace{0.2cm}

\section{Convergence analysis under the PL condition}\label{sec-PL}

In this section, we show the linear convergence result under the PL condition.

\vspace{0.1cm}

\begin{proof}[Proof of Theorem \ref{thm-pl}]
~Using the $L$-smoothness of $f$, we obtain that for each $t \geq \tau$,
\begin{equation}\label{pl-f-eq1}
f(x_{t+1}) - f(x_t) \leq - \eta \langle \nabla f(x_t), \nabla f(x_{t-\tau}) \rangle + \frac{L\eta^2}{2} \|\nabla f(x_{t-\tau})\|^2.
\end{equation}
One part of \eqref{pl-f-eq1} can be rewritten as
\begin{equation*}
\langle \nabla f(x_t), \nabla f(x_{t-\tau}) \rangle = \frac{1}{2} \|\nabla f(x_t)\|^2 + \frac{1}{2}\|\nabla f(x_{t-\tau})\|^2 - \frac{1}{2} \|\nabla f(x_t) - \nabla f(x_{t-\tau})\|^2,
\end{equation*}
and then the inequality \eqref{pl-f-eq1} becomes
\begin{equation}\label{eq-4-14}
f(x_{t+1}) -f(x_t) \leq - \frac{\eta}{2}\|\nabla f(x_t)\|^2 -\frac{1}{2}(\eta - L\eta^2) \|\nabla f(x_{t-\tau})\|^2 + \frac{\eta}{2}\|\nabla f(x_t) - \nabla f(x_{t-\tau})\|^2.
\end{equation}
Here, in order to estimate the last term in \eqref{eq-4-14}, we observe that if $\tau+1 \leq t \leq 2 \tau$, then we have $x_{t-\tau}  =x_{\tau}$ and
\begin{equation*} 
x_{t}  = x_{\tau} - \eta \sum_{j=0}^{t-\tau-1} \nabla f(x_j).    
\end{equation*}
Also, if $t \geq 2 \tau +1$, then
\begin{equation*}
 x_t = x_{\tau} - \eta \sum_{j=0}^{t-\tau-1} \nabla f(x_j) \quad \textrm{and}\quad 
x_{t-\tau} = x_{\tau} - \eta \sum_{j=0}^{t-2\tau-1}\nabla f(x_j),
\end{equation*}
so we have
\begin{equation*}
x_t -x_{t-\tau} = -\eta \sum_{j=t-2\tau}^{t-\tau-1} \nabla f(x_j).
\end{equation*}
By the $L$-smoothness \eqref{L-smooth} and combining the above results, one gets
\begin{equation*}
\begin{split}
\|\nabla f(x_t) - \nabla f(x_{t-\tau})\|&\leq L \|x_t - x_{t-\tau}\|
\\
&=\left\{\begin{array}{ll} L \eta  \|\sum_{j=t-\tau}^{t-1}\nabla f(x_{j-\tau})\|&\quad\textrm{ if }\,t \geq 2 \tau+1,
\\
L \eta \|\sum_{j= \tau}^{t-1}\nabla f(x_{j-\tau})\|&\quad\textrm{ if }\,\tau +1 \leq t \leq 2 \tau,
\\
 0&\quad\textrm{ if }\,t=\tau.  \end{array}
\right.
\end{split}
\end{equation*} 
Using the Cauchy-Schwartz inequality, we have
\begin{equation}\label{pl-f-eq2}
\|\nabla f(x_t)-\nabla f(x_{t-\tau})\|^2 \leq \eta^2 L^2 \tau \mathcal{R}_1(t),
\end{equation}
where
\begin{equation*}
\mathcal{R}_1(t) := \left\{\begin{array}{ll}   \sum_{j=t-\tau}^{t-1}\|\nabla f(x_{j-\tau})\|^2&\quad\textrm{ if }\,t \geq 2 \tau+1,
\\
  \sum_{j= \tau}^{t-1}\|\nabla f(x_{j-\tau})\|^2 &\quad\textrm{ if }\,\tau +1 \leq t \leq 2 \tau,
\\
 0&\quad\textrm{ if }\,t=\tau.  \end{array}
\right.
\end{equation*}
By \eqref{pl-f-eq2}, the inequality \eqref{eq-4-14} becomes
\begin{equation}\label{pl-f-eq3}
\begin{split}
&f(x_{t+1}) -f(x_t) 
\\
&\leq - \frac{\eta}{2}\|\nabla f(x_t)\|^2 -\frac{1}{2}(\eta - L\eta^2) \|\nabla f(x_{t-\tau})\|^2 + \frac{1}{2}\eta^3 L^2 \tau \mathcal{R}_1(t)
\\
&\leq - {\eta \mu} \Big( f(x_t) -f_*\Big) -\frac{1}{2}(\eta - L\eta^2) \|\nabla f(x_{t-\tau})\|^2 + \frac{1}{2}\eta^3 L^2 \tau  \mathcal{R}_1(t).
\end{split}
\end{equation}
Rearranging \eqref{pl-f-eq3}, we have the following estimate: for $t \geq \tau$,
\begin{equation}\label{eq-4-11}
\begin{split}
 f(x_{t+1}) -f_*  &\leq \left(1-  {\eta \zeta}\right)\Big( f(x_t) -f_*\Big) -\frac{1}{2}(\eta - L\eta^2) \|\nabla f(x_{t-\tau})\|^2   + \frac{1}{2}\eta^3 L^2 \tau \mathcal{R}_1(t).
\end{split}
\end{equation}
For simplicity, we set
$$
a_t := f(x_{t+\tau} ) -f_*,\qquad b_t :=\|\nabla f(x_t)\|^2.
$$
The sequential inequality \eqref{eq-4-11} can be rewritten as follows: For $t\geq 0$,
\begin{equation}\label{pl-a-eq20}
a_{t+1}  \leq \left\{\begin{array}{ll}c_1 a_0 - Q_1 b_0 &\quad \textrm{for } ~t=0,
\\
c_1 a_t +\delta_1 ( b_{t-1} + \cdots + b_0) - Q_1 b_t&\quad \textrm{for }~ 1 \leq t \leq \tau,
\\
c_1 a_t + \delta_1 (b_{t-1} + \cdots + b_{t-\tau}) - Q_1 b_t, &\quad \textrm{for }~t \geq \tau+1,
\end{array}\right.
\end{equation}
where
\begin{equation*}
\begin{aligned}
c_1:= 1- {\eta \zeta},\qquad \delta_1 := \frac{\eta^3 L^2 \tau}{2},\qquad Q_1 := \frac{1}{2}(\eta -L\eta^2).
\end{aligned}
\end{equation*}
To apply Lemma \ref{lem-2-2} to \eqref{pl-a-eq20}, we need to proceed to check the following condition:
\begin{equation*}
\delta_1 \left( \frac{1-c_1^j}{1-c_1}\right) \leq c_1^j Q_1,
\end{equation*}
which is reduced to
\begin{equation}\label{eq-4-12}
\frac{\eta^3 L^2 \tau}{2} \cdot \frac{1-c_1^j}{\eta \zeta}\leq c_1^j Q_1.
\end{equation}
If the above inequality holds, then we can apply Lemma \ref{lem-2-2} to \eqref{eq-4-11}, which is yielding the following estimate:
\begin{equation*}
f(x_{t+\tau}) -f_* \leq (1-\eta\zeta)^{t} \Big(f(x_{ \tau}) -f_*\Big).
\end{equation*}
Actually, this is the desired inequality, so we only need to check the condition \eqref{eq-4-12} for completing the proof.

Now, we show \eqref{eq-4-12} which is equivalent to
\begin{equation}\label{pl-eta-eq1}
\frac{\eta^2 L^2 \tau}{2\zeta} \leq \left( Q_1 + \frac{\eta^2 L^2 \tau}{2\zeta}\right) c_1^j.
\end{equation}
Since the right hand side of \eqref{pl-eta-eq1} decreases in $j$ due to $c_1<1$, it is sufficient to show \eqref{pl-eta-eq1} for only $j=\tau$.
From Lemma \ref{lem-2-1}, we recall that
\begin{equation}\label{pl-c-eq1}
c_1^{-\tau}  = \frac{1}{(1-\eta\zeta)^{\tau}}\leq 1 + J_{\tau}  {\eta \zeta \tau},
\end{equation}
provided that $\eta \zeta \tau \leq 1/5$.
If we choose $\eta>0$ satisfying
\begin{equation*}
\eta \leq \frac{2}{\left(\sqrt{1+4 J_{\tau}  \tau^2}+1\right)L},
\end{equation*}
then we have
\begin{equation*}
\eta^2 + \frac{\eta}{ J_{\tau} L \tau^2} \leq \frac{1}{ J_{\tau} L^2 \tau^2},
\end{equation*}
which is equivalent to
\begin{equation}\label{pl-J-eq1}
J_{\tau}  {\eta \zeta \tau}  \leq \frac{\zeta (1-L\eta)}{\eta L^2 \tau}.
\end{equation}
By \eqref{pl-J-eq1}, the inequality \eqref{pl-c-eq1} becomes
\begin{equation}\label{eq-4-10}
c_1^{-\tau} \leq 1 + \frac{\zeta(1-L\eta)}{\eta L^2 \tau}.
\end{equation}
Rearranging \eqref{eq-4-10}, we obtain \eqref{pl-eta-eq1} for $j=\tau$.
Hence, the required condition \eqref{eq-4-12} holds for $\eta>0$ satisfying \eqref{eq-1-10}, and then the proof is done.
\end{proof}

\vspace{0.2cm}

\section{Extension to stochastic gradient descent with time-varying delay}\label{sgd}

{In this section, we consider the stochastic gradient descent with time-varying delay. Namely, we consider the following algorithm: For a number $\boldsymbol{\tau}\in \mathbb{N}$, with $x_0 = x_1 = \cdots = x_{\boldsymbol{\tau}-1} =x_{\boldsymbol{\tau}}$,
\begin{equation}\label{eq-7-1}
x_{t+1} = x_t - \eta \left( \nabla f\left(x_{t-\tau (t)}\right) + \xi_{t}\right)\qquad \forall\,t \geq  \boldsymbol{\tau},
\end{equation}
where $\xi_{t}$  is a noise of zero mean and $\tau (t)$ is a time-varying delay with $0 \leq \tau (t) \leq  \boldsymbol{\tau}$. This algorithm has been studied by many authors (see \cite{AD, LSZ, NBB}), since it appears naturally for analyzing the parallel or distributed computing of the stochastic gradient descent. We consider the following assumption on the noise (cf. \cite{SK}):
\begin{ass}[$(M,\, \sigma^2)$-bounded noise]\label{ass-3} There exist two constants $M,\, \sigma^2 \geq 0$ such that 
\begin{equation}\label{eq-7-2}
\mathbb{E}\left[\,\xi_t \,|\,x_t\,\right] =\mathbf{0} \in \mathbb{R}^d,\qquad \mathbb{E}\left[\,\|\xi_t\|^2 \,|\,x_t\,\right] \leq M \left\|\nabla f\left(x_{t-\tau (t)}\right)\right\|^2 + \sigma^2.
\end{equation}
\end{ass}
We refer to \cite{AHSL} for a justification of the first condition of \eqref{eq-7-2}.  We will use the notation $\mathbb{E}_t$ to denote the expectation regarding the distribution of $\xi_{t}$ and  also use the notation $\mathbb{E}$ to denote the expectation over the whole time, i.e., $\mathbb{E} = \mathbb{E}_N \mathbb{E}_{N-1} \cdots \mathbb{E}_1 \mathbb{E}_0$, where $N$ denotes a large time that the iteration of \eqref{eq-7-1} terminates.

The convergence property of \eqref{eq-7-1} with constant time delay $\tau (t) \equiv \boldsymbol{\tau}$ was studied in the literatures \cite{ASS, SK}, when {\it Assumption} \ref{ass-3}: \eqref{eq-7-2} holds and the cost function is convex or strongly convex. In the following theorem, we extend the argument used in Theorem \ref{thm-pl} to obtain a convergence result of the algorithm \eqref{eq-7-1} under the PL condition.
\begin{thm}\label{thm-5-1}Suppose that $f:\mathbb{R}^d\rightarrow\mathbb{R}$ is $L$-smooth and it satisfies the PL condition \eqref{eq-PL}, and the noise $\xi_t$ in \eqref{eq-7-1} satisfies the condition \eqref{eq-7-2}.
If we assume that the stepsize $\eta>0$ and the bound of delay $\boldsymbol{\tau}>0$ satisfy the following inequality:
\begin{equation}\label{eq-7-10}
\eta \leq \min\left\{ \frac{L(1+M)}{5 \zeta},~ \frac{2\boldsymbol{\tau}}{\sqrt{1+4 J_{\boldsymbol{\tau}}  \boldsymbol{\tau}^2}+1 }\right\}\frac{1}{L(1+M)\boldsymbol{\tau}},
\end{equation}
then the algorithm \eqref{eq-7-1} satisfies the following estimate: For any $t \geq \boldsymbol{\tau}$,
\begin{equation}
\mathbb{E}\Big[ f(x_{t}) - f_* \Big] \leq   (1-\eta \zeta)^{t-\boldsymbol{\tau}} \left(f (x_{\boldsymbol{\tau}})-f_*\right) +  \left( \frac{\boldsymbol{\tau}^2 \eta^2 L^2}{2\zeta}+\frac{L\eta}{2\zeta}\right)\sigma^2.
\end{equation}
\end{thm}
\begin{proof}
For simplicity, we denote $g_t = \nabla f \left(x_{t-\tau (t)}\right) + \xi_{t}$ in the algorithm \eqref{eq-7-1}. Using the $L$-smoothness property of $f$, we have
\begin{equation}
f(x_{t+1}) - f(x_t) \leq -\eta \left\langle \nabla f(x_t),~g_{t}\right\rangle + \frac{L \eta^2}{2} \|g_t\|^2\qquad \forall\,t \geq \boldsymbol{\tau}.
\end{equation}
Taking the expectation $\mathbb{E}_t$ regarding the distribution of $\xi_t$ and using \eqref{eq-7-2} here, one yields
\begin{equation}\label{eq-7-5}
\begin{split}
&\mathbb{E}_t \Big[ f(x_{t+1}) - f(x_t)\Big] 
\\
& \leq - \eta \langle \nabla f(x_t), ~\nabla f(x_{t-\tau (t)})\rangle + \frac{L\eta^2}{2} \mathbb{E}_t\|g_t\|^2
\\
& = - \frac{\eta}{2} \Big( \|\nabla f(x_t)\|^2 + \|\nabla f(x_{t-\tau (t)})\|^2 - \|\nabla f(x_t) - \nabla f(x_{t-\tau (t)})\|^2\Big) + \frac{L\eta^2}{2} \mathbb{E}_t \|g_{t}\|^2.
\end{split}
\end{equation}
By \eqref{eq-7-2}, we get
\begin{equation}\label{eq-7-3}
\mathbb{E}_t \|g_t\|^2 =  \|\nabla f(x_{t-\tau (t)})\|^2  + \mathbb{E}_t\|\xi_{t}\|^2 \leq (1+M)\|\nabla f(x_{t-\tau (t)})\|^2   +\sigma^2.
\end{equation}
Using the $L$-smoothness property of $f$ and the fact that $\tau (t) \leq \boldsymbol{\tau}$, we also have the following estimate:
\begin{equation}
\begin{split}
\left\|\nabla f(x_t) - \nabla f\left(x_{t-\tau (t)}\right)\right\|&\leq L \left\|x_t - x_{t-\tau (t)}\right\|
\\
&\leq L \sum_{k=0}^{\boldsymbol{\tau} -1} \left\|x_{t-k} -x_{t-k-1}\right\|
\\
&=\left\{\begin{array}{ll} L \eta  \sum_{j=t-\boldsymbol{\tau}}^{t-1}\|g_j\|&\quad\textrm{ if }\,t \geq 2 \boldsymbol{\tau}+1,
\\
L \eta  \sum_{j= \boldsymbol{\tau}}^{t-1}\|g_j\|&\quad\textrm{ if }\,\boldsymbol{\tau} +1 \leq t \leq 2 \boldsymbol{\tau},
\\
 0&\quad\textrm{ if }\,t=\boldsymbol{\tau}.  \end{array}\right.
\end{split}
\end{equation}
Combining this with \eqref{eq-7-3}, one has  for $t \geq 2 \tau +1$ the following estimate 
{\begin{equation}
\begin{split}
\mathbb{E}_t \left\|\nabla f(x_t) - \nabla f\left(x_{t-\tau (t)}\right)\right\|^2 &\leq \boldsymbol{\tau} \eta^2 L^2 \sum_{k=t-\boldsymbol{\tau}}^{t-1} \mathbb{E}_t\|g_t\|^2
\\
&\leq \boldsymbol{\tau} \eta^2 L^2 \sum_{k=t-\boldsymbol{\tau}}^{t-1}(1+M)\left\|\nabla f\left(x_{k-\tau (k)}\right)\right\|^2 + \boldsymbol{\tau}^2 \eta^2 L^2 \sigma^2.
\end{split}
\end{equation}}
The above estimate can be obtained with the range of the summation $\sum_{k=t-\boldsymbol{\tau}}^{t-1}$ replaced by  $\sum_{k=\boldsymbol{\tau}}^{t-1}$ and $0$ for the cases $\boldsymbol{\tau}+1 \leq t \leq 2 \boldsymbol{\tau}$ and $t=\boldsymbol{\tau}$, respectively.

Inserting this estimate and \eqref{eq-7-3} into \eqref{eq-7-5} for the case $t \geq 2\boldsymbol{\tau} +1$ gives
\begin{equation} 
\begin{split}
&\mathbb{E}_t\Big[ f(x_{t+1}) - f(x_t)\Big] 
\\
& \leq  - \frac{\eta}{2}   \|\nabla f(x_t)\|^2  -\left(\frac{\eta}{2} - \frac{L(1+M)\eta^2}{2}\right) \left\|\nabla f\left(x_{t-\tau (t)}\right)\right\|^2 + \frac{L\eta^2 \sigma^2}{2}
\\
&\qquad +\frac{\boldsymbol{\tau} \eta^3 L^2 (1+M)}{2} \sum_{k=t-\boldsymbol{\tau}}^{t-1}\left\|\nabla f\left(x_{k-\tau (k)}\right)\right\|^2 + \frac{\boldsymbol{\tau}^2 \eta^3 L^2 \sigma^2}{2}
\\
& \leq- {\eta}\zeta   (f(x_t) -f_*)  -\left(\frac{\eta}{2} - \frac{L(1+M)\eta^2}{2}\right) \left\|\nabla f\left(x_{t-\tau (t)}\right)\right\|^2 \\
&\qquad + \frac{\boldsymbol{\tau} \eta^3 L^2  (1+M)^2}{2} \sum_{k=t-\boldsymbol{\tau}}^{t-1}\left\|\nabla f\left(x_{k-\tau (k)}\right)\right\|^2
+\left( \frac{\boldsymbol{\tau}^2 \eta^3 L^2  \sigma^2}{2}+\frac{L\eta^2 \sigma^2}{2}\right).
\end{split}
\end{equation}
Since $f(x_{t+1}) -f(x_t) = (f(x_{t+1})-f_*) - (f(x_t)-f_*)$, and taking the expectation over the whole $t \geq \boldsymbol{\tau}$ in the above estimate, we get
\begin{equation}\label{eq-7-6} 
\begin{split}
&\mathbb{E}\Big[ f(x_{t+1}) - f_*\Big] 
\\
& \leq(1-  {\eta}\zeta)\,\mathbb{E}\Big[ f(x_{t}) - f_*\Big]  -\left(\frac{\eta}{2} - \frac{L(1+M)\eta^2}{2}\right) \mathbb{E}\left\|\nabla f\left(x_{t-\tau (t)}\right)\right\|^2 
\\
&\qquad+ \frac{\boldsymbol{\tau} \eta^3 L^2 (1+M)^2}{2} \sum_{k=t-\boldsymbol{\tau}}^{t-1}\mathbb{E}\left\|\nabla f\left(x_{k-\tau (k)}\right)\right\|^2 +\left( \frac{\boldsymbol{\tau}^2 \eta^3 L^2 \sigma^2}{2}+\frac{L\eta^2 \sigma^2}{2}\right).
\end{split}
\end{equation}
The above inequality holds with the range of the summation $\sum_{k=t-\boldsymbol{\tau}}^{t-1}$ replaced by  $\sum_{k=\boldsymbol{\tau}}^{t-1}$ and $0$ for the cases $\boldsymbol{\tau}+1 \leq t \leq 2 \boldsymbol{\tau}$ and $t=\boldsymbol{\tau}$, respectively.

In order to find a bound of $\mathbb{E}\left[ f(x_{t}) - f_*\right]$ from the above inequality, we consider two sequences $\{a_{t}\}_{t\geq \boldsymbol{\tau}}$ and $\{w_t\}_{t\geq \boldsymbol{\tau}}$ as follows: With $a_{\boldsymbol{\tau}} = f(x_{\boldsymbol{\tau}})-f_*$ and $w_{\boldsymbol{\tau}} =0$, the sequences are defined by
\begin{equation*}
a_{t+1} = \left(1-\eta \zeta\right) a_t -\left(\frac{\eta}{2}-\frac{L(1+M)\eta^2}{2}\right)b_t + \frac{\boldsymbol{\tau} \eta^3 L^2 (1+M)^2}{2}\left\{\begin{array}{ll}\sum_{k=t-\boldsymbol{\tau}}^{t-1} b_k &\textrm{if}\,\,t \geq 2\boldsymbol{\tau}+1
\medskip
\\
\sum_{k=\boldsymbol{\tau}}^{t-1} b_k &\textrm{if}\,\, \boldsymbol{\tau}+1\leq t \leq  2\boldsymbol{\tau}
\medskip
\\
0 &\textrm{if}\,\, t =\boldsymbol{\tau}
\end{array}\right.
\end{equation*}
where $b_k := \mathbb{E} \left\|\nabla f\left(x_{k-\tau(k)}\right)\right\|^2$ and 
\begin{equation}\label{eq-7-8}
w_{t+1} = \left(1-\eta \zeta\right) w_t +\left( \frac{\boldsymbol{\tau}^2 \eta^3 L^2 \sigma^2}{2}+\frac{L\eta^2 \sigma^2}{2}\right)\qquad \forall\,t \geq \boldsymbol{\tau}.
\end{equation}
Then the sequence $\{(a_t +w_t)\}_{t \geq 0}$ satisfies the equality case of \eqref{eq-7-6} with $\mathbb{E}[f(x_t)-f_*]$ replaced by $(a_t +w_t)$, and the initial conditions are same, i.e., $(a_{\boldsymbol{\tau}} +w_{\boldsymbol{\tau}}) = f(x_{{\boldsymbol{\tau}}})-f_*$. Therefore, we have
\begin{equation}\label{eq-7-7}
\mathbb{E}\Big[ f(x_{t}) - f_*\Big] \leq a_t +w_t\qquad \forall\,t \geq \boldsymbol{\tau}.
\end{equation}
The sequence $\{a_t\}_{t\geq \boldsymbol{\tau}}$ satisfies the same sequential inequality \eqref{pl-a-eq20} with the sequence $\{a_t\}_{t \geq \boldsymbol{\tau}}$ defined in the proof of  Theorem \ref{thm-pl} with $L$ replaced by  $L(1+M)$ and $\tau$ replaced by $\boldsymbol{\tau}$. Therefore, the same analysis applies to yield the following bound:
\begin{equation}
a_{t} \leq (1-\eta \zeta)^{t-{\boldsymbol{\tau}}}a_{{\boldsymbol{\tau}}},
\end{equation}
provided that the stepsize $\eta>0$ satisfies \eqref{eq-7-10}.

As for $w_{t}$, we solve the equation \eqref{eq-7-8} recursively to find for $t \geq {\boldsymbol{\tau}}+1$,
\begin{equation}
\begin{split}
w_t & = (1-\eta \zeta)^{t-{\boldsymbol{\tau}}} w_{{\boldsymbol{\tau}}} + \sum_{j=0}^{t-{\boldsymbol{\tau}}-1} (1-\eta \zeta)^{j}\left( \frac{{\boldsymbol{\tau}}^2 \eta^3 L^2 \sigma^2}{2}+\frac{L\eta^2 \sigma^2}{2}\right)
\\
& \leq (1-\eta \zeta)^{t-{\boldsymbol{\tau}}} w_{{\boldsymbol{\tau}}}  + \sum_{j=0}^{\infty} (1-\eta \zeta)^{j} \left( \frac{{\boldsymbol{\tau}}^2 \eta^3 L^2 \sigma^2}{2}+\frac{L\eta^2 \sigma^2}{2}\right)
\\
&= (1-\eta \zeta)^{t-{\boldsymbol{\tau}}} w_{{\boldsymbol{\tau}}}  +  \left( \frac{{\boldsymbol{\tau}}^2 \eta^2 L^2}{2\zeta}+\frac{L\eta}{2\zeta}\right)\sigma^2.
\end{split}
\end{equation}
Combining the above two estimates in \eqref{eq-7-7}, we finally get the following estimate:
\begin{equation}
\mathbb{E}\Big[ f(x_{t}) - f_* \Big] \leq   (1-\eta \zeta)^{t-{\boldsymbol{\tau}}} (f (x_{{\boldsymbol{\tau}}})-f_*) +  \left( \frac{{\boldsymbol{\tau}}^2 \eta^2 L^2}{2\zeta}+\frac{L\eta}{2\zeta}\right)\sigma^2.
\end{equation}
The proof of this theorem is finished.
\end{proof}

}

\vspace{0.2cm}

\section{Numerical experiments}\label{sec-4}

In this section, we firstly try to confirm the analyzed result in Theorem \ref{thm-1-5} regarding the gradient descent with delay in two settings: least-squares regression and logistic classification (cf. \cite{SBBLMO}).
Moreover, we show some numerical examples satisfying the PL inequality in order to confirm Theorem \ref{thm-pl}.

\subsection{Least-squares Regression}

Consider the regularized least-squares regression problem satisfying the optimization problem: $\min_{x\in\mathbb{R}^d}f(x)$, where
\begin{equation}\label{fx-eq-4-1}
f(x):=\frac{1}{m}\sum_{i=1}^m\left(y_i-A_i \cdot x\right)^2+\frac{\mu}{2}\|x\|^2,
\end{equation}
where $A_i\in\mathbb{R}^d$ is a vector containing the $i$-th data point, and $y_i\in\mathbb{R}$ is the $i$-th associated value.
In this experiment, we choose $m=1,000$ data points with Gaussian random variables $A_i\sim\mathcal{N}({\bf 0},{\bf 1})$ of mean ${\bf 0}$ and variance ${\bf 1}$ in the fixed setting $d=10$.
Also, we set $y_i=A_i \cdot {\bf 1}+\cos\left( A_i \cdot {\bf 1}\right)+\xi_i$, where $\xi_i\sim\mathcal{N}(0,1/4)$ is an i.i.d. Gaussian noise of variance $1/4$.

When the cost function $f(x)$ is given by \eqref{fx-eq-4-1} with $\mu=0.1$, we simulate the concerned gradient descent scheme \eqref{eq-1-1} with the delay $\tau=5$, $10$, $20$, $100$, and then we try to confirm the behavior of the error $\|x_t-x_*\|$ in Theorem \ref{thm-1-5}.
{The minimizer $x_*\in\mathbb{R}^d$ of the given $f(x)$ can be written as
$$
x_*=\left(\sum_{i=1}^m A_i A_i^T + \frac{m\mu}{2}I\right)^{-1}\left(\sum_{i=1}^m A_i \,y_i\right),
$$
where $I\in\mathbb{R}^{d\times d}$ is the identity matrix.
We set the initial point $x_0=0$ and for each iteration $t\geq \tau$, we define the log-scaled error $\mathcal{E}_t$ as
\begin{equation}\label{err-def}
\mathcal{E}_t:=\ln\left(\frac{\|x_t -x_*\|}{\|x_0 - x_*\|}\right).
\end{equation}

\begin{figure}[htbp]
\begin{center}
\subfigure[$\tau=5$]{
\includegraphics[height=5.9cm]{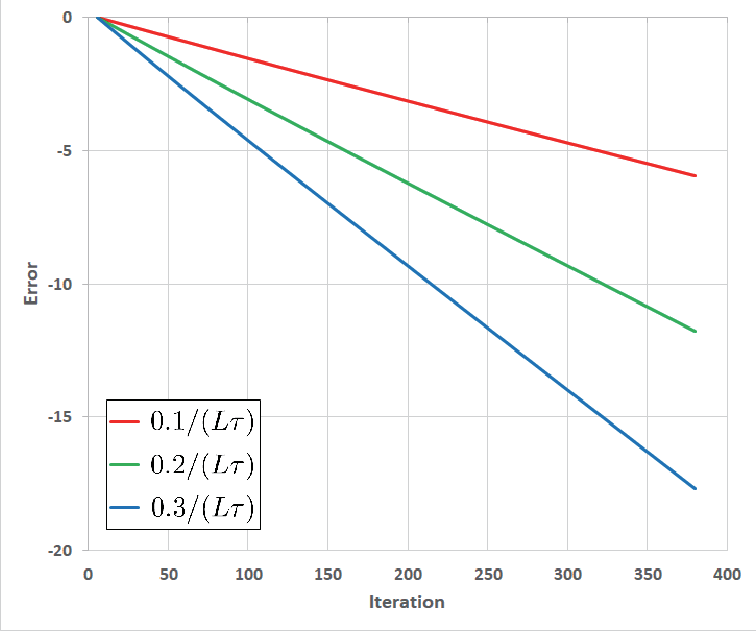}}
\quad
\subfigure[$\tau=10$]{
\includegraphics[height=5.9cm]{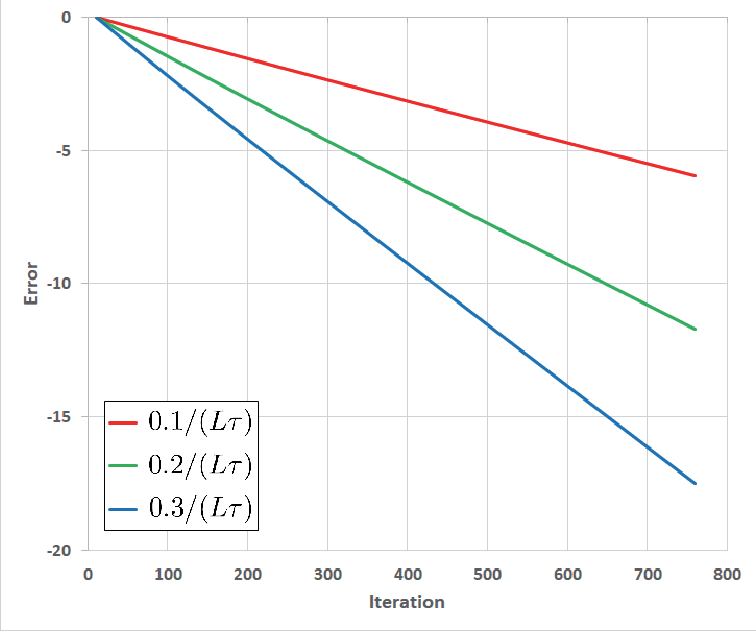}}\\
\subfigure[$\tau=20$]{
\includegraphics[height=5.9cm]{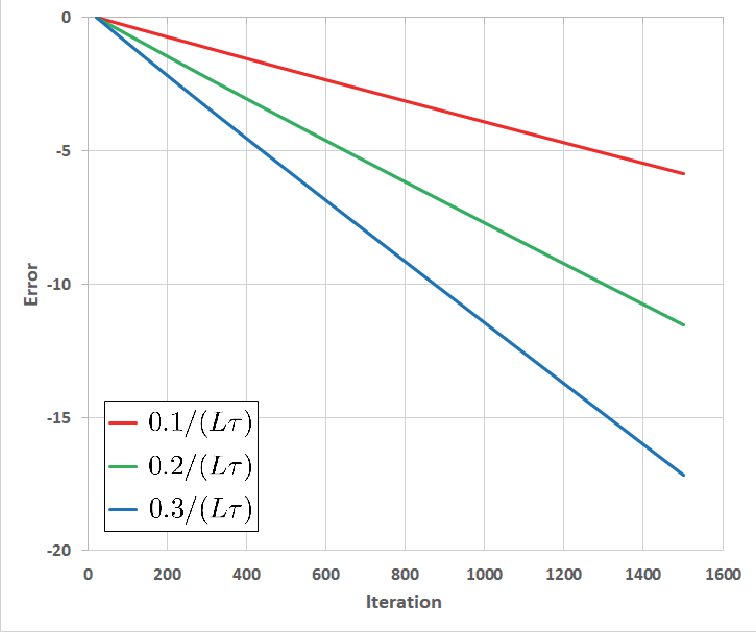}}
\quad
\subfigure[$\tau=100$]{
\includegraphics[height=5.9cm]{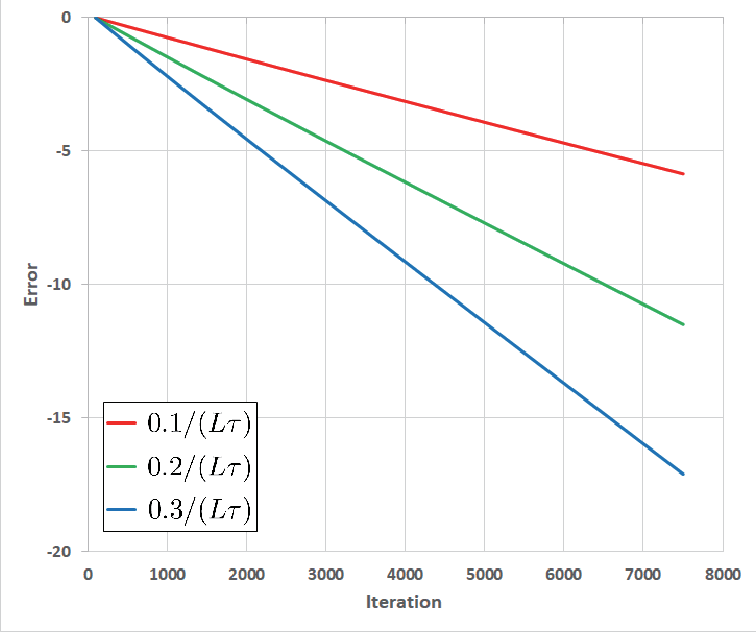}}
\end{center}
\caption{Graphs of the log-scaled error $\mathcal{E}_t$ computed by the gradient descent with various delay $\tau=5$, $10$, $20$, $100$, regarding the least-squares regression problem.}\label{fig1}
\end{figure}

On Figure \ref{fig1}, we depict some errors $\mathcal{E}_t$ variously simulated by setting $\eta=0.1/(L\tau)$, $0.2/(L\tau)$, $0.3/(L\tau)$ for each delay $\tau$. These stepsizes satisfy the condition \eqref{eq-1-30} of \mbox{Theorem \ref{thm-1-1}} in view of the values $C_{\tau}$ computed in Table \ref{values}.
Since
$$
\begin{aligned}
\|\nabla f(y)-\nabla f(x)\|&=\left\|\left(\frac{2}{m}\sum_{i=1}^m A_i A_i^T+\mu I\right)(y-x)\right\|\\
&\leq\left\|\,\frac{2}{m}\sum_{i=1}^m A_i A_i^T+\mu I\,\right\|\|y-x\|,
\end{aligned}
$$
the constant $L>0$ can be obtained by the largest singular value of $\frac{2}{m}\sum_{i=1}^m A_i A_i^T+\mu I$.
}
Seeing all graphs in Figure \ref{fig1}, we observe that the numerical values of $\mathcal{E}_t$ decay linearly with respect to $t$ for any delay $\tau$ as expected by Theorem \ref{thm-1-5}.
 Compared with the number of iterations for each delay $\tau$, one additionally sees that the convergence of $\|x_t-x_*\|$ is faster  when $\tau$ is chosen smaller.

\subsection{Logistic Classification}

Consider the logistic classification problem: $\min_{x\in\mathbb{R}^d}f(x)$, where
\begin{equation}\label{fx-eq-4-2}
f(x) := \frac{1}{m}\sum_{i=1}^m \ln \left(1 + e^{-y_i A_i\cdot x}\right) + \frac{\mu}{2} \|x\|^2,
\end{equation}
where $A_i\in\mathbb{R}^d$ is a vector containing the $i$-th data point, and $y_i\in\{-1,1\}$ is the $i$-th associated class assignment.
In this experiment, we try to choose $m=1,000$ data points: $500$ data points for the first class and $500$ data points for the second class. Each data point $A_i\sim\mathcal{N}(y_i{\bf 1},{\bf 1})$ is the Gaussian random variable of mean $y_i{\bf 1}$ and variance ${\bf 1}$, where
$$
y_i=\left\{
\begin{aligned}
&1&&~~\mbox{ if }~1\leq i\leq 500,\\
&-1&&~~\mbox{ otherwise}.
\end{aligned}
\right.
$$

Using the cost function $f(x)$ of \eqref{fx-eq-4-2} with $\mu=0.1$, we try to show the behavior of the error $\|x_t-x_*\|$ by simulating the concerned gradient descent scheme \eqref{eq-1-1} with the delay $\tau=5$, $10$, $20$, $100$.
{Actually, it is difficult to find the minimizer $x_*\in\mathbb{R}^d$ of the given $f(x)$.
To measure the convergence of error, we set the initial point $x_0=0$ and the minimizer $x_*\in\mathbb{R}^d$ is approximately found by \eqref{eq-1-1} within the tolerance ${\rm TOL}=10^{-10}$.

\begin{figure}[htbp]
\begin{center}
\subfigure[$\tau=5$]{
\includegraphics[height=5.9cm]{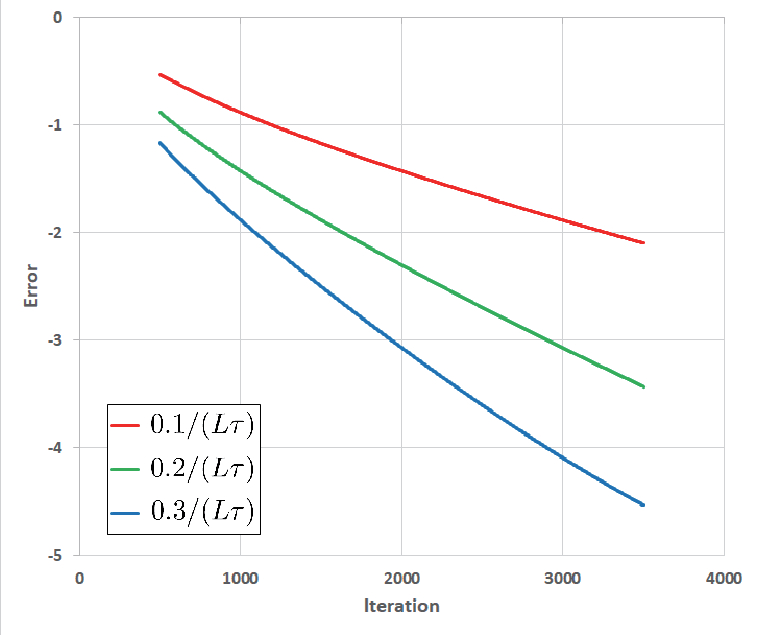}}
\quad
\subfigure[$\tau=10$]{
\includegraphics[height=5.9cm]{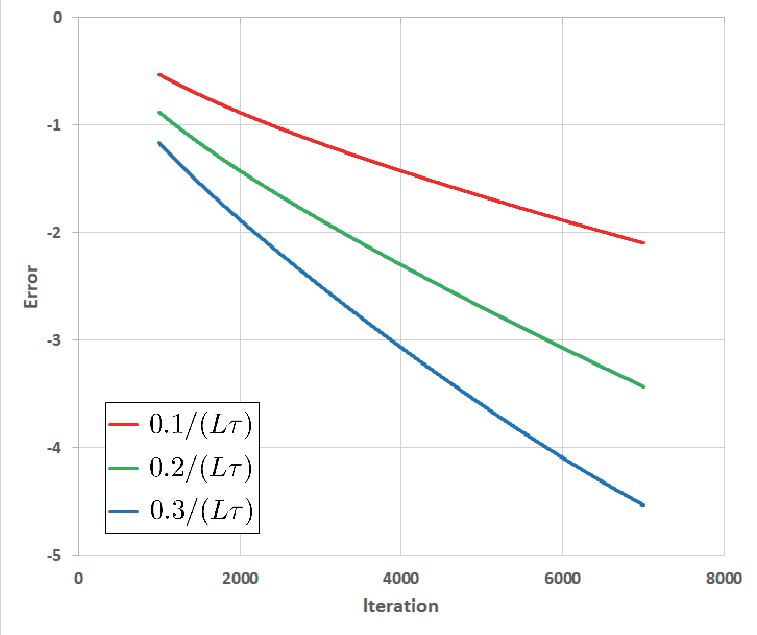}}\\
\subfigure[$\tau=20$]{
\includegraphics[height=5.9cm]{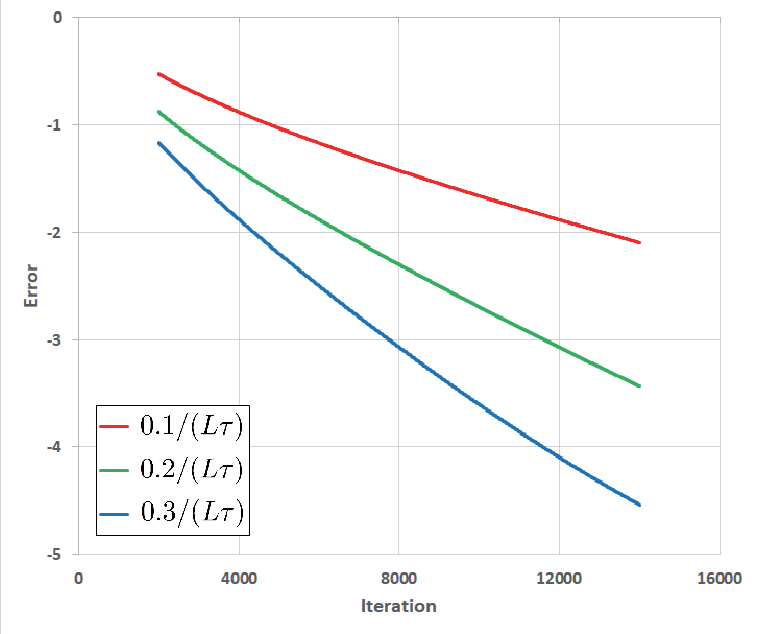}}
\quad
\subfigure[$\tau=100$]{
\includegraphics[height=5.9cm]{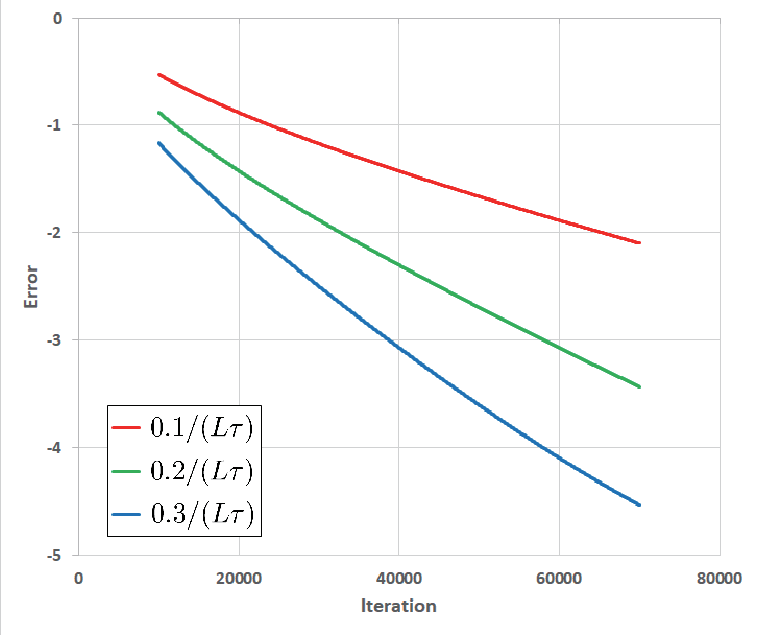}}
\end{center}
\caption{Graphs of the log-scaled error $\mathcal{E}_t$ computed by the gradient descent with various delay $\tau=5$, $10$, $20$, $100$, regarding the logistic classification problem.}\label{fig2}
\end{figure}

On Figure \ref{fig2}, we show the graphs of $\mathcal{E}_t$ obtained by \eqref{err-def} for each delay $\tau=5$, $10$, $20$, $100$ with various learning rate $\eta=0.1/(L\tau)$, $0.2/(L\tau)$, $0.3/(L\tau)$.
In this experiment, the constant $L$ can be computed by
\begin{equation}\label{experiment2-L}
L=\frac{1}{m}\sum_{i=1}^m \left\|A_i A_i^T\right\|+\mu.
\end{equation}
Here, the reason of \eqref{experiment2-L} is as follows: Since
$$
\begin{aligned}
\nabla f(x)&=-\frac{1}{m}\sum_{i=1}^m\left(\frac{y_i \,e^{-y_i A_i\cdot x}}{1+e^{-y_i A_i\cdot x}}\right)A_i+\mu x,\\
\nabla^2 f(x)&=\frac{1}{m}\sum_{i=1}^m\left(\frac{y_i^2 \, e^{-y_i A_i\cdot x}}{\left(1+e^{-y_i A_i\cdot x}\right)^2}\right)A_i A_i^T +\mu I,
\end{aligned}
$$
we have
$$
\begin{aligned}
\|\nabla f(y)-\nabla f(x)\|&=\left\|\,\int_0^1 \frac{d}{ds}\left(\nabla f(x+s(y-x))\right)ds\,\right\|\\
&=\left\|\,\int_0^1 \nabla^2 f(x+s(y-x))\,(y-x)\,ds\,\right\|\\
&\leq\int_0^1 \left\|\nabla^2 f(x+s(y-x))\right\|\left\|y-x\right\|ds\\
&\leq\int_0^1 \left(\frac{1}{m}\sum_{i=1}^m \left\|A_i A_i^T\right\|+\mu\right)\left\|y-x\right\|ds\\
&\leq L\|y-x\|.
\end{aligned}
$$
}
From all the graphs in Figure \ref{fig2}, we observe the linear decay property of  $\mathcal{E}_t$ which supports  the result of Theorem \ref{thm-1-5}.

\subsection{Numerical example satisfying PL condition}

In this subsection, we show the reliability of the analyzed result in Theorem \ref{thm-pl}.
Consider the following cost function:
\begin{equation}\label{f-def}
f(x) = \frac{1}{2} \|Ax -b\|^2,
\end{equation}
where $A \in \mathbb{R}^{m\times d}$ and $b \in \mathbb{R}^m$. If $d>m$, the cost $f(x)$ of \eqref{f-def} is not strongly convex, but it satisfies PL inequality for any $A$ and $b$ (refer to \cite{KNS}).
To construct the numerical example satisfying PL condition, we set $d=15$ and $m=6$, and moreover, all elements of $A$ and $b$ are independently generated by the Gaussian random variable of mean $0$ and variance $1$.

With the cost $f(x)$ given by \eqref{f-def}, we try to simulate the gradient descent \eqref{eq-1-1} with the delay $\tau=25$, based on the initial point $x_0=0$.
{Since
$$
\|\nabla f(y)-\nabla f(x)\|=\|A^T A(y-x)\|\leq \|A^T A\|\,\|y-x\|,
$$
the constant $L>0$ can be obtained by the largest singular value of $A^T A$.
 
Choosing a sufficiently small number {$\zeta>0$} used in \eqref{eq-PL}, the range of learning rate $\eta>0$ assumed in Theorem \ref{thm-pl} becomes
\begin{equation}\label{eq-5-1}
\eta \leq \frac{D_{\tau}}{L\tau} \quad \textrm{with} \quad D_{\tau} := \frac{2\tau}{\sqrt{1+4 J_{\tau}\tau^2}+1} \simeq  {0.9305.}
\end{equation}
More precisely, we want the following equality in  the range \eqref{eq-1-10} of the stepsize $\eta$ of Theorem \ref{thm-pl}:
\begin{equation}
\min \Big\{ \frac{L}{5\zeta}, ~\frac{2\tau}{\sqrt{1+4J_{\tau} \tau^2 }+1}\Big\}   =\frac{2\tau}{\sqrt{1+4J_{\tau} \tau^2 }+1}.
\end{equation}
Since the value $\frac{2\tau}{\sqrt{1+4J_{\tau} \tau^2 }+1} \simeq 0.9305$ for $\tau =25$ and $L = \sigma_{max}(AA^T)$, it is enough to choose  $\zeta$ such that
\begin{equation}
\frac{\sigma_{max}(AA^T)}{5\zeta} \geq \frac{2\tau}{\sqrt{1+4J_{\tau} \tau^2 }+1} \simeq 0.9305.
\end{equation}}
For each iteration $t\geq \tau$, we define the log-scaled cost function error $e_t$ as
\begin{equation}\label{error-et}
e_t:=\ln\left(\frac{f(x_t)-f_*}{f(x_\tau)-f_*}\right),
\end{equation}
where $f_*$ denotes the minimal value of the given cost $f$, i.e., $f_*=f(x_*)$.
Actually, the minimizer $x_*\in\mathbb{R}^d$ can be obtained by $x_*=A^+ b$, where $A^+$ is the pseudo-inverse of $A$.

\begin{figure}[htbp]
\begin{center}
\subfigure[$\eta=\frac{C}{L\tau}$ for $C\leq 1.05$ and $\eta=\frac{1}{20L(\tau+1)}$]{
\includegraphics[height=5.5cm]{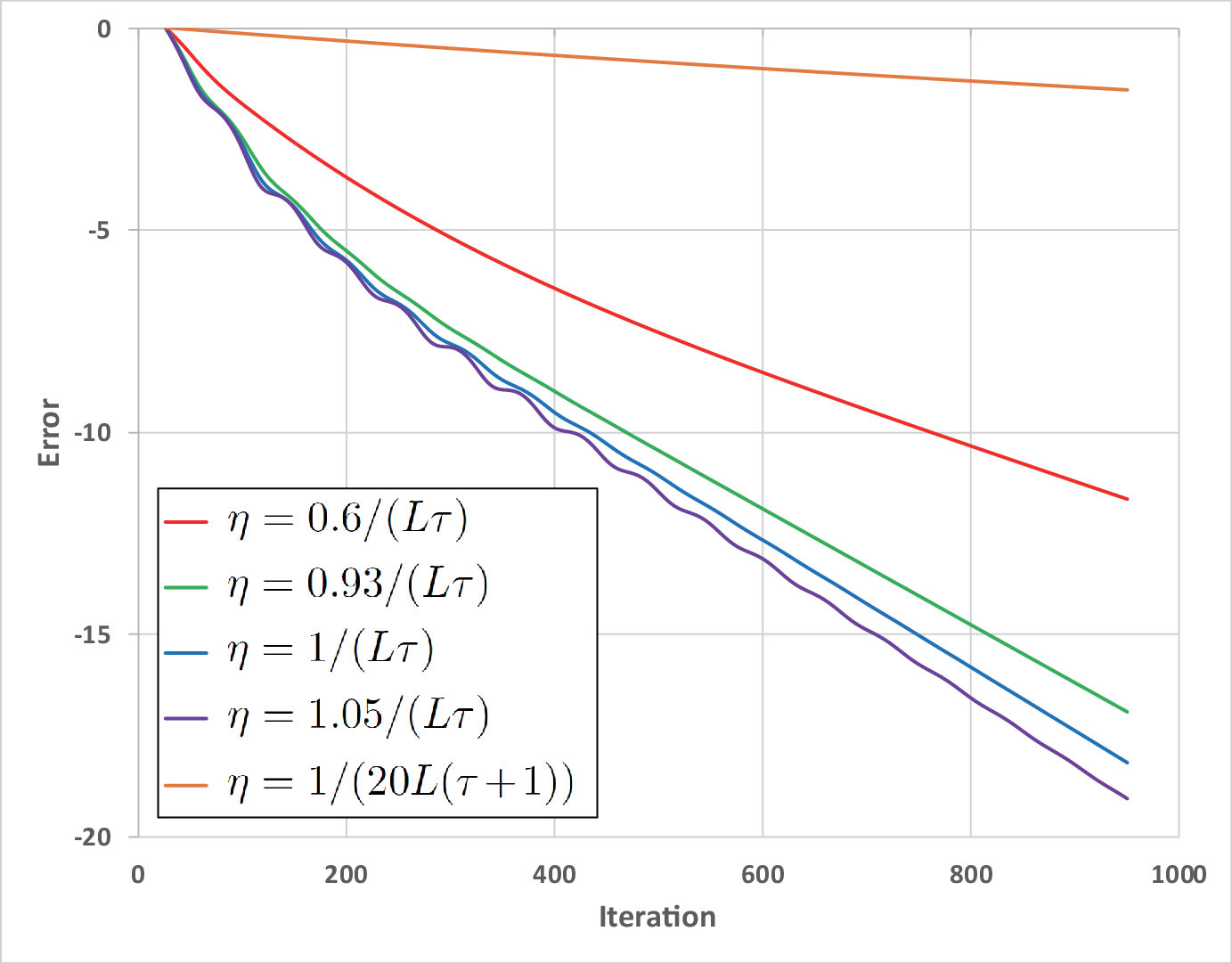}\label{fig3-a}}\\
\subfigure[$\eta=\frac{C}{L\tau}$ for $1.1\leq C\leq 1.4$]{
\includegraphics[height=5.5cm]{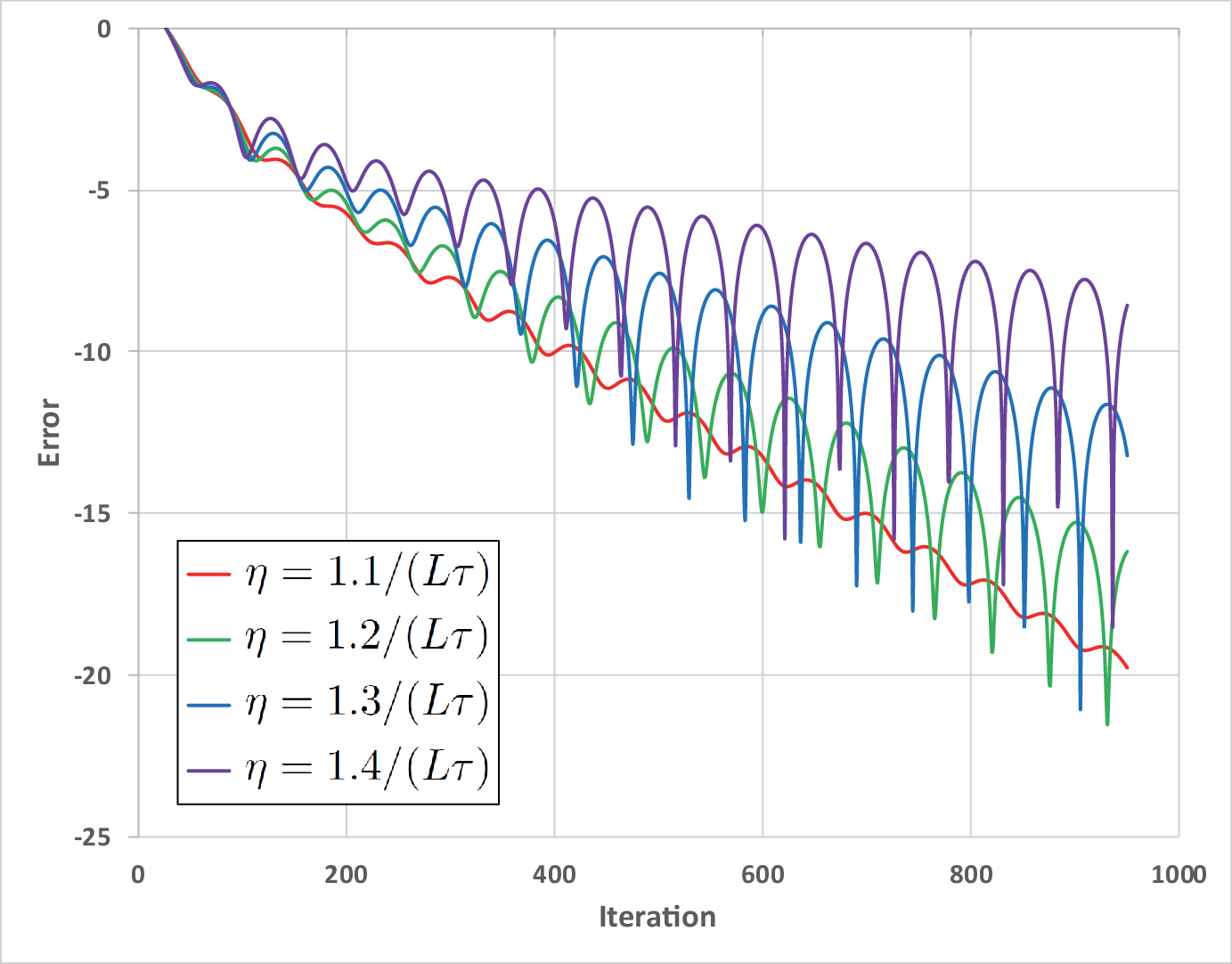}\label{fig3-b}}
\quad
\subfigure[$\eta=\frac{C}{L\tau}$ for $C\geq 1.5$]{
\includegraphics[height=5.5cm]{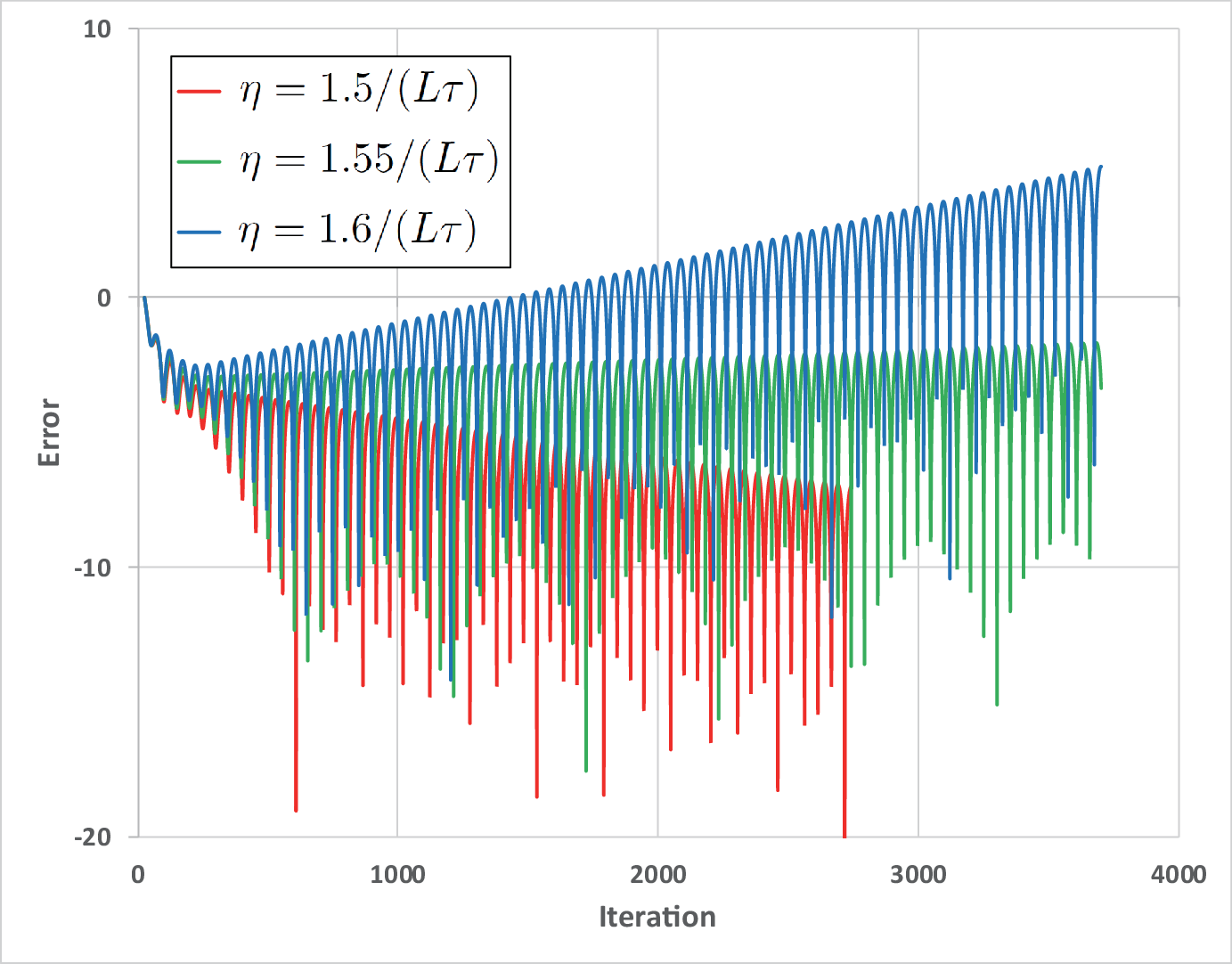}\label{fig3-c}}
\end{center}
\caption{Graphs of the log-scaled cost error $e_t$ for the cost function $f(x)=2^{-1}\|Ax-b\|^2$, computed by the gradient descent with the delay $\tau=25$.}\label{fig3}
\end{figure}

{On Figure \ref{fig3}, we present the graphs of $e_t$ with respect to the iteration $t$ for some learning rates such as $\eta=C/(L\tau)$ for some $C>0$ and $\eta=1/(20L(\tau+1))$. Here, the case of $\eta=1/(20L(\tau+1))$ is mentioned in \cite{ASS}.
We remark that the assumption \eqref{eq-5-1} is satisfied in two cases: $\eta=0.6/(L\tau)$ and $0.93/(L\tau)$.
Seeing the graphs of Figure \ref{fig3-a}, it is observed that each error $e_t$ for $\eta=0.6/(L\tau)$ and $0.93/(L\tau)$ monotonically decays with linear rate, which is equivalent to the analyzed result \eqref{analyzed_PL} in Theorem \ref{thm-pl}. Actually, the monotonically linearly decreasing property of $e_t$ only holds for $\eta\leq 1/(L\tau)$, which implies that the analyzed condition of learning rate in Theorem \ref{thm-pl} may be almost necessary and sufficient to guarantee the linear rate of $e_t$.
On the other hand, the oscillation of $e_t$ for $\eta\geq 1.05/(L\tau)$ can be found on Figure \ref{fig3}, which means that the error $e_t$ is not monotonically decreasing.
Furthermore, one sees that the error $e_t$ can blow up for a large learning rate $\eta$ on Figure \ref{fig3-c}.
}

{\subsection{Numerical example for stochastic gradient descent}

Here, we provide a numerical experiment for the stochastic gradient descent with time-varying delay. We consider the logistic classification problem \eqref{fx-eq-4-2} with the Wisconsin breast cancer dataset from UCI machine learning repository \cite{DG}. Letting $f_i(x):=\ln\left(1+e^{-y_i \,A_i\cdot x}\right)+\frac{\mu}{2}\|x\|^2$, we implement the following mini-batch stochastic gradient descent with time-varying delay:
\begin{equation}\label{eq-6-1}
x_{t+1} = x_t - \eta \left(\frac{1}{B} \sum_{k=1}^B \nabla f_{n_k(t)} \left(x_{t-\tau (t)}\right)\right)\qquad\mbox{ for }\,t \geq \boldsymbol{\tau},
\end{equation}
where $B \in \mathbb{N}$ is the batch size, $n_k(t)$ is a number randomly selected from $\{1,\,2,\,\cdots,\, m\}$ for each $1 \leq k \leq B$ and $t \geq \boldsymbol{\tau}$, and $\tau(t)$ is a time-varying delay randomly selected from $\{0,\,1,\,\cdots,\,\boldsymbol{\tau}\}$.
When $B=20$ or $100$, we simulate the algorithm \eqref{eq-6-1} with constant stepsizes $\eta=1.5\times 10^{-2}/\boldsymbol{\tau}$ and $6.0\times 10^{-2}/\boldsymbol{\tau}$ for iterations up to $10^6$ and delay bounds $\boldsymbol{\tau}=10$ and $50$.
In view of Theorem \ref{thm-5-1}, if we want to get an exact convergence of \eqref{eq-6-1} to the optimal value, a reasonable choice is to repeat diminishing the stepsize (cf. \cite{LZL}).
So, in this experiment, we also consider a cascading stepsize, that is, we repeat reducing the stepsize by half on some certain iterations.
For $\boldsymbol{\tau}=10$, we choose $\eta = 6.0\times 10^{-2}/\boldsymbol{\tau}$ for the initial stepsize, and repeat reducing the stepsize by half on each iteration $t\in\{1.0\times 10^5,\, 1.5\times 10^5,\, 3.5\times 10^5\}$.
Similarly, for $\boldsymbol{\tau}=50$, we choose the same initial stepsize $\eta = 6.0\times 10^{-2}/\boldsymbol{\tau}$ and repeat reducing the stepsize by half on each iteration $t \in \{4.0\times 10^5,\, 4.5\times 10^5,\, 6.5\times 10^5\}$.

On Figure \ref{fig4}, we show the error $e_t$ measured by \eqref{error-et} for each iteration $t\geq \boldsymbol{\tau}$.
One sees that the error sequence $\{e_t\}$ computed by \eqref{eq-6-1} with each constant stepsize converges exponentially fast up to a certain error whose size depends on the used stepsize.
On the other hand, the sequence $\{x_t\}$ simulated by a cascading stepsize continues to approach to the optimal value.

\begin{figure}[htbp]
\begin{center}
\subfigure[$\boldsymbol{\tau}=10$ (Batch size: $20$)]{
\includegraphics[height=5.5cm]{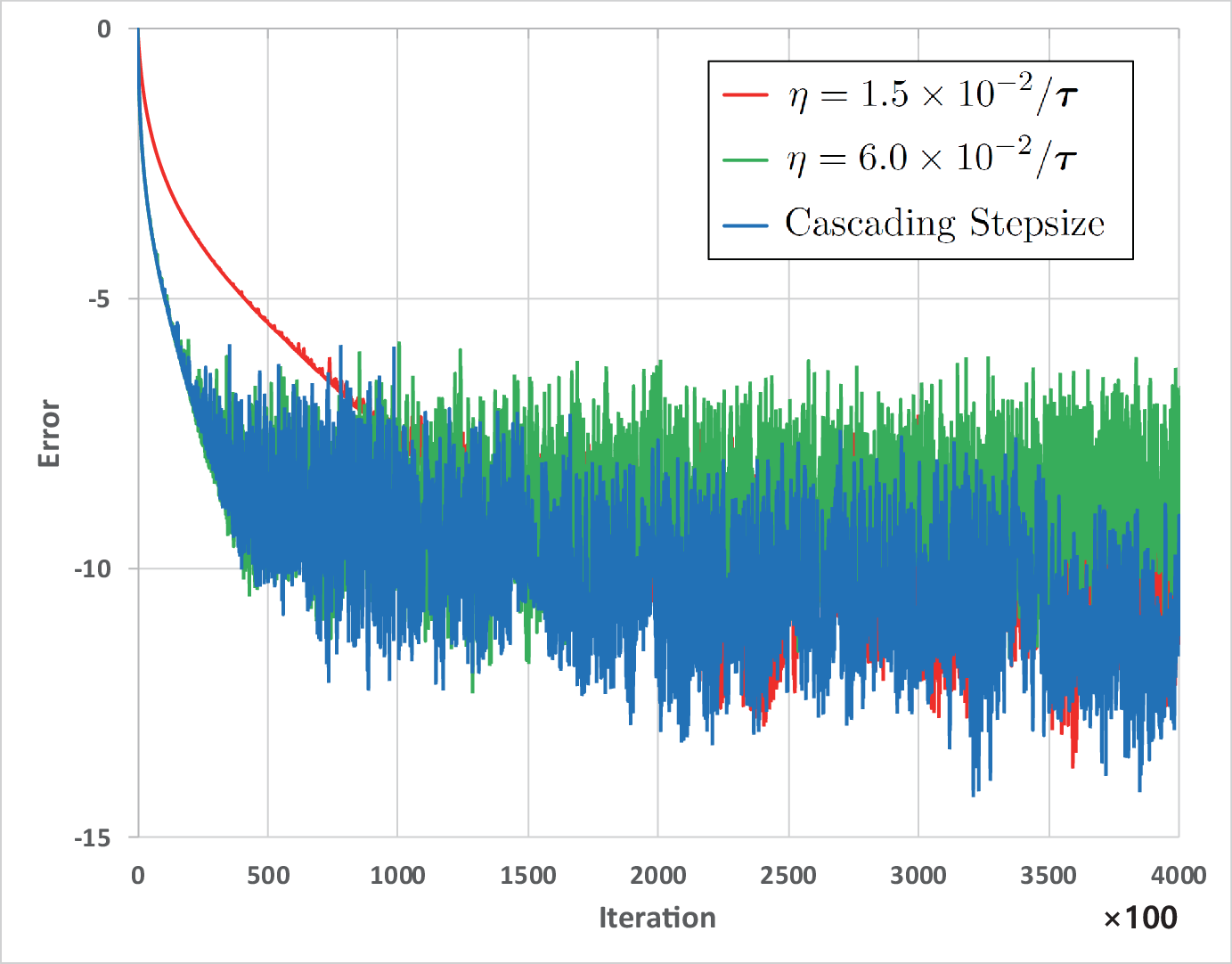}}
\quad
\subfigure[$\boldsymbol{\tau}=50$ (Batch size: $20$)]{
\includegraphics[height=5.5cm]{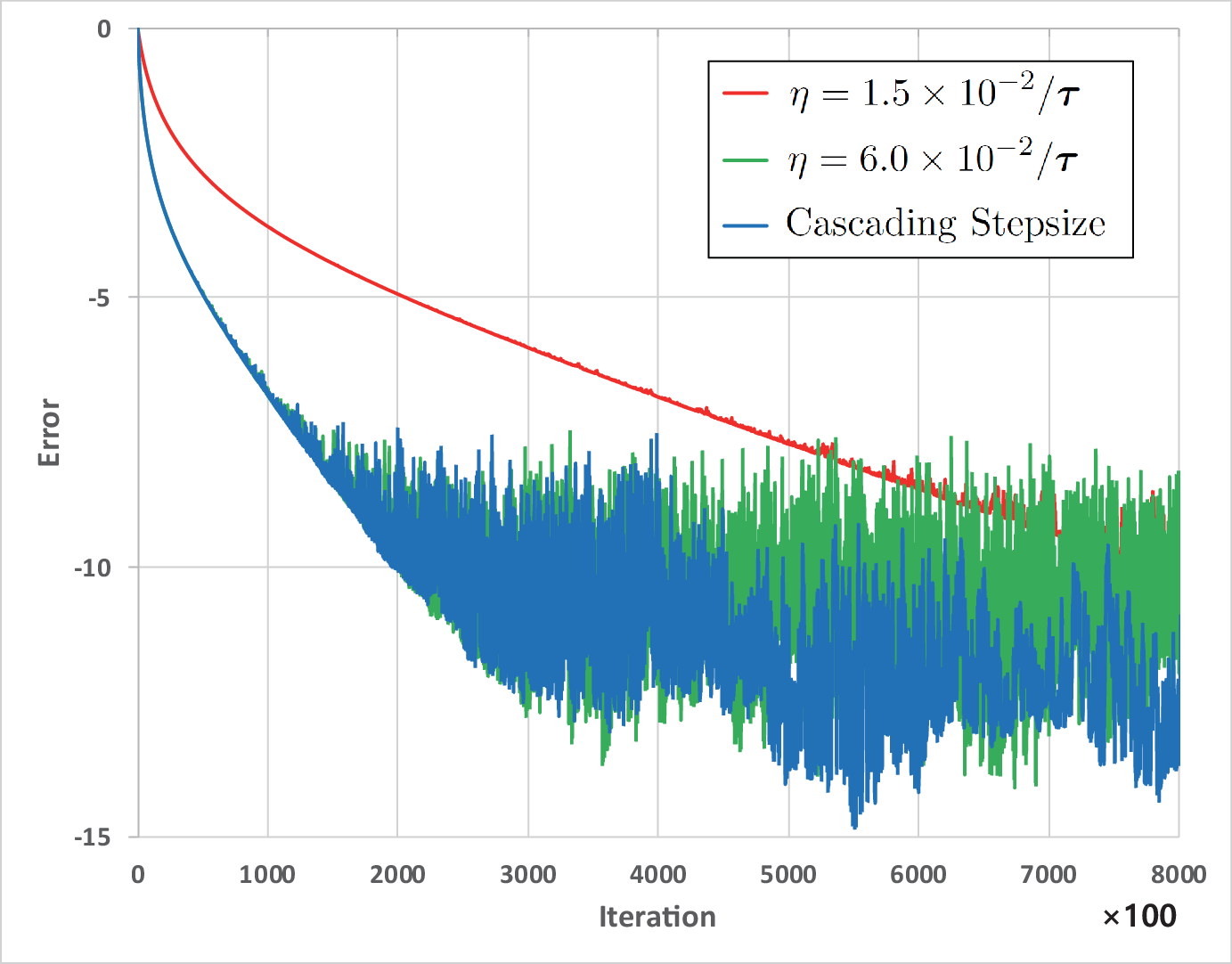}}\\
\subfigure[$\boldsymbol{\tau}=10$ (Batch size: $100$)]{
\includegraphics[height=5.5cm]{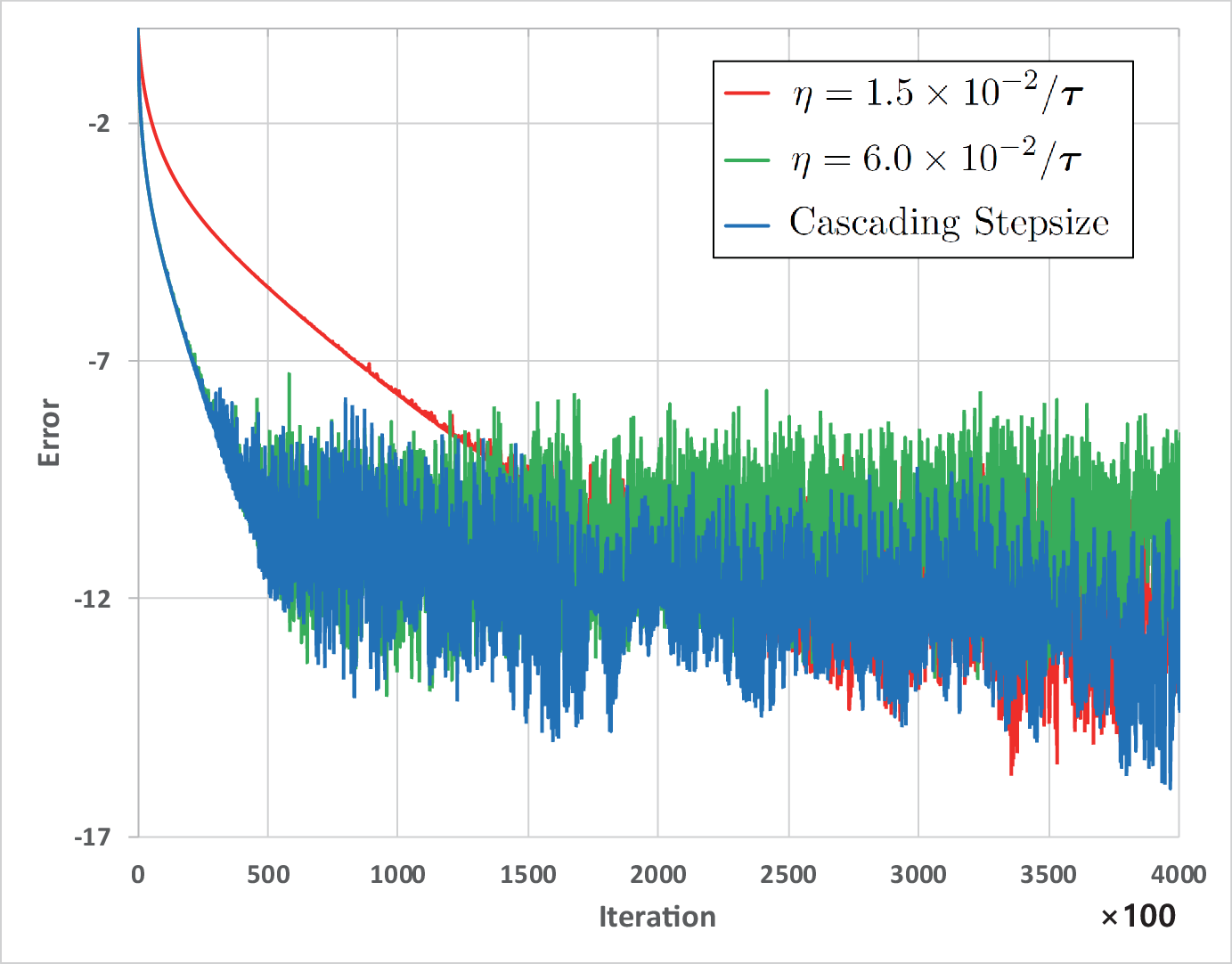}}
\quad
\subfigure[$\boldsymbol{\tau}=50$ (Batch size: $100$)]{
\includegraphics[height=5.5cm]{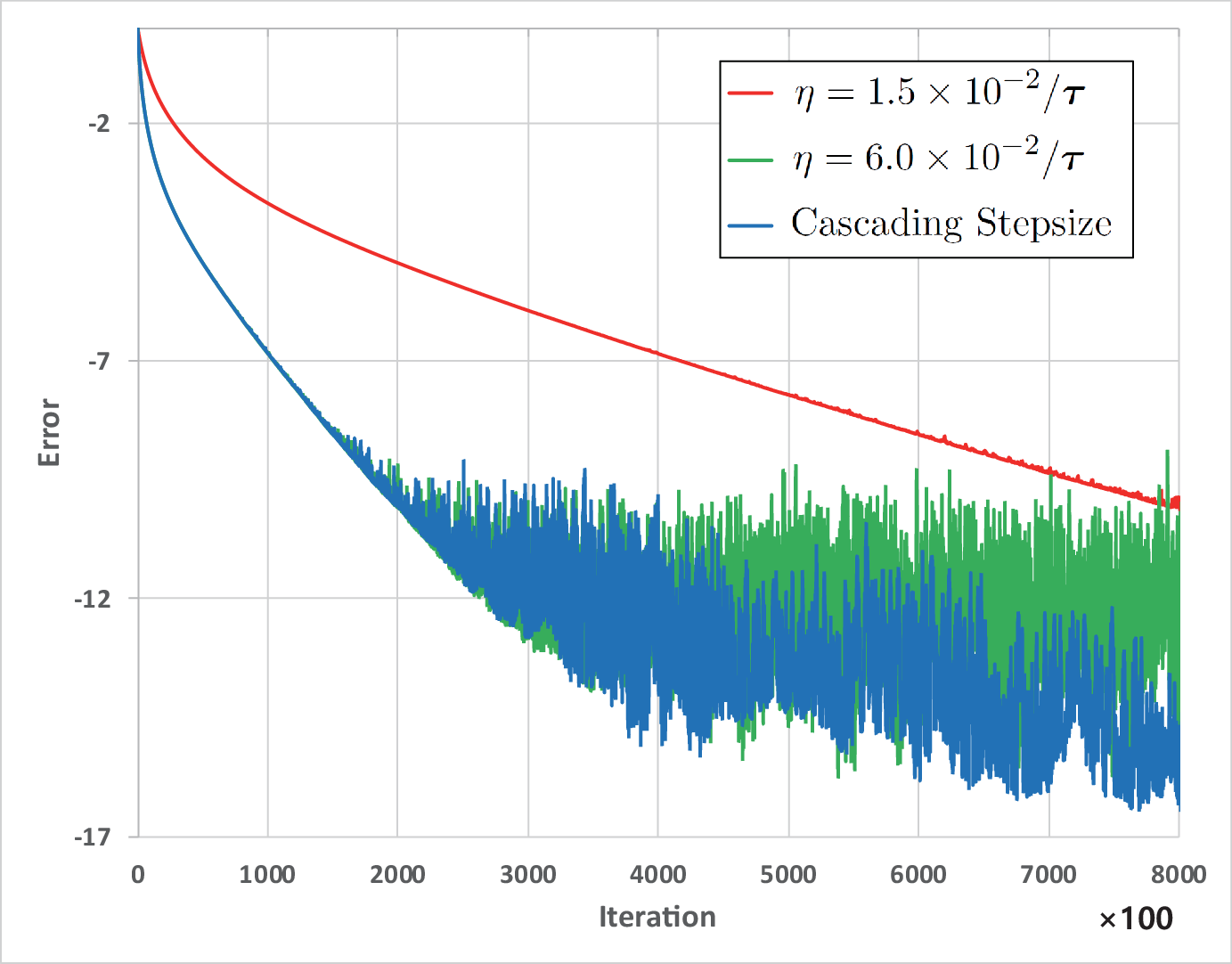}}
\end{center}
\caption{Graphs of the log-scaled cost error $e_t$ computed by the mini-batch stochastic gradient descent with various delay bounds $\boldsymbol{\tau}=10$, $50$.}\label{fig4}
\end{figure}
}

\vspace{0.2cm}

\section*{Conclusion} 
This paper establishes the linear convergence estimate for the gradient descent involving the delay, when the given cost function has the strongly convexity and $L$-smoothness. In addition, we give the analysis of linear convergence for the cost function error under the Polyak-{\L}ojasiewicz condition and extend the 
result to the stochastic gradient descent with time-varying delay.
The results are confirmed  by numerical experiment.

\vspace{0.2cm}

\section*{Acknowledgments}
The authors thank the anonymous referees for their helpful comments that improved the quality of the manuscript. H. J. Choi was supported by the National Research Foundation of Korea (grant RS-2023-00280065).
W. Choi was supported by  the National Research Foundation of Korea (grant NRF-2021R1F1A1059671).
\mbox{J. Seok} was supported by  the National Research Foundation of Korea (grant NRF-2020R1C1C1A01006415).

\end{document}